\documentclass[preprint,12pt]{elsarticle}

\usepackage{amssymb}

\usepackage{calc}
\usepackage[czech, american]{babel}
\usepackage[utf8]{inputenc}
\usepackage{csquotes}

\usepackage{amsthm, amsmath}
\usepackage[pdftex,dvipsnames]{xcolor}
\usepackage[procnames]{listings}
\usepackage{makecell}

\theoremstyle{definition}
\newtheorem{definition}{Definition}
\newtheorem{corollary}[definition]{Corollary}
\newtheorem{remark}[definition]{Remark}
\newtheorem{example}[definition]{Example}
\newtheorem{example*}{Example}

\theoremstyle{plain}
\newtheorem{theorem}[definition]{Theorem}
\newtheorem{proposition}[definition]{Proposition}
\newtheorem{lemma}[definition]{Lemma}
\newtheorem{observation}[definition]{Observation}

\newtheorem{assumption}[definition]{Assumption}

\newcommand{\overbar}[1]{\mkern 1.5mu\overline{\mkern-1.5mu#1\mkern-1.5mu}\mkern 1.5mu}
\newcommand{\ba}{\overbar{a}}
\newcommand{\T}{\vartheta}
\newcommand{\DT}{(\Delta, \Theta)}
\newcommand{\mbu}{\mathbf{u}}
\newcommand{\N}{\mathbb{N}}
\newcommand{\A}{\mathcal{A}}

\newcounter{prefixrule}

\newenvironment{prefixrule}{\refstepcounter{prefixrule}\equation}{\tag{\arabic{prefixrule}}\endequation}
\usepackage{hhline}

\journal{European Journal of Combinatorics}

\begin{document}

\begin{frontmatter}



\title{Normalization of ternary generalized pseudostandard words}


\author{Josef Florian}
\ead{pepca.florian@gmail.com}
\author{Tereza Veselá}
\ead{tereza.velka@gmail.com}
\author{Ľubomíra Dvořáková}
\ead{lubomira.balkova@gmail.com}
\address{Department of Mathematics, Faculty of Nuclear Sciences and Physical Engineering, Czech Technical University in Prague, 13 Trojanova, 12000 Praha 2, Czech Republic}

\begin{abstract}
This paper focuses on generalized pseudostandard words, defined by de Luca and De Luca in 2006. In every step of the construction, the involutory antimorphism to be applied for the pseudopalindromic closure changes and is given by a so called directive bi-sequence. The concept of a~normalized form of directive bi-sequences was introduced by Blondin-Massé et al. in 2013 and an algorithm for finding the normalized directive bi-sequence over a binary alphabet was provided.
In this paper, we present an algorithm to find the normalized form of any directive bi-sequence over a ternary alphabet. Moreover, the algorithm was implemented in Python language and carefully tested, and is now publicly available in a module for working with ternary generalized pseudostandard words.

\end{abstract}


\end{frontmatter}


\section{Introduction}
\label{Introduction}
This paper focuses on generalized pseudostandard words. Such words were defined by de Luca and De Luca in 2006~\citep{deluca} as a generalization of standard episturmian words.  In every step of the construction, the involutory antimorphism to be applied for the pseudopalindromic closure changes and is given by a so called \textit{directive bi-sequence}.
While standard episturmian and pseudostandard words have been studied intensively and a~lot of their properties are known (see for instance~\citep{BuLuZa,droubay2,deluca2,deluca}), only little has been shown so far about generalized pseudostandard words.

In~\citep{deluca} the authors defined generalized pseudostandard words and proved there that the famous Thue--Morse word is an example of such words.
Jajcayov\'a et al.~\citep{multipseudo} characterized generalized pseudostandard words in the class of generalized Thue--Morse words.
Jamet et al.~\citep{paquin} dealt with fixed points of the palindromic and pseudopalindromic closure and formulated an open problem concerning fixed points of the generalized pseudopalindromic closure. The first and the third author of this paper provided a~necessary and sufficient condition on the periodicity of binary and ternary generalized pseudostandard words in~\citep{pepa2} and studied complexity and formulated a new conjecture on complexity of binary generalized pseudostandard words in~\citep{pepakomplexita}.
The second and the third autor of this paper found a~new class of fixed points of morphisms among binary generalized pseudostandard words and formulated a~conjecture concerning such fixed points in~\citep{nasclanek}.
Binary generalized pseudostandard words were primarily studied by Blondin-Massé et al. \citep{bm}, the following results were obtained for instance:
\begin{itemize}
	\item The concept of a~\textit{normalized form} of a directive bi-sequence was introduced. Such a form can be found for every generalized pseudostandard word and has some additional useful properties compared to a non-normalized directive bi-sequence.
	\item A necessary and sufficient condition to decide if a directive bi-sequence is normalized over a binary alphabet was provided.
\item An algorithm to find the normalized form of any directive bi-sequence was presented.
\end{itemize}

In this paper, we generalize the results from~\citep{bm} to a ternary alphabet in the following sense:
\begin{itemize}
\item We introduce an algorithm to find the normalized form of any directive bi-sequence over a ternary alphabet. The algorithm for the ternary alphabet turns out to be much more complex than in the binary case.
\item The algorithm was implemented in Python language and carefully tested, and is now available in a module for working with ternary generalized pseudostandard words.
\end{itemize}

The paper is organized as follows. In Section~\ref{Preliminaries}, we first introduce the definitions and notations from combinatorics on words used in the sequel, we recall what generalized pseudostandard words are, and mention some of their properties. Section \ref{Normalization} is devoted to the normalized form of ternary directive bi-sequences. Section \ref{Implementation} summarizes the key aspects of the implementation of the normalization algorithm. In the last section, we summarize open problems concerning generalized pseudostandard words.

\section{Preliminaries}\label{Preliminaries}

A finite non-empty set $\A$ of symbols is called an \textit{alphabet}, the symbols are called \textit{letters}. A finite (infinite) \textit{word} $\mbu$ is a finite (infinite) sequence of letters. The length $|w|$ of a finite word $w$ is the number of letters it contains. The concatenation of two words $u = u_1u_2 \ldots u_n$ and $v = v_1 v_2 \ldots v_m$ is the word $uv = u_1 \ldots u_n v_1 \ldots v_m$. The neutral element for concatenation of words is the \textit{empty word} $\varepsilon$ and its length is set to $|\varepsilon| = 0$. The set of all finite non-empty words over an alphabet $\A$ is $\A^+$, if we add the empty word, then $\A^*$. The symbol $\A ^{\N}$ denotes the set of infinite words over an alphabet $\A$.

The \textit{factor} of an infinite, resp. a finite word $\mbu$ is a finite word $w \in \A^*$ such that $\mbu = pws$, where $p$ is a finite word and $s$ is an infinite, resp. a finite word. The factor $p$ is called a \textit{prefix} and the word $s$ a \textit{suffix}. If $|p| = |s|$ and $\mbu$ is finite, then $w$ is a central factor of $\mbu$. A factor of $\mbu$ is called \textit{proper} if it is not equal to the whole word $\mbu$. Let $u,v,w$ be three words such that $w = uv$. The word $wv^{-1}$ is the word $w$ without its suffix $v$, i.e., $wv^{-1} = u$. If $w = u x^{-1} v$ for some non-empty word $x$, then we say that $u$ and $v$ \textit{overlap} and $x$ is their \textit{overlap}.




\subsection{Involutory antimorphisms and pseudopalindromes}
An \textit{involutory antimorphism} is a map $\vartheta : \A^* \to \A^*$ such that for every $u, v$ $\in \A^*$ we have $\vartheta (uv) = \vartheta (v) \vartheta (u)$ and $\vartheta ^2$ is the identity map. Any antimorphism is given if the letter images are provided, i.e., $\vartheta (a)$ for every $a \in \A$. This work will focus on the binary alphabet $\A = \{0,1\}$ and the ternary alphabet $\A=\{0,1,2\}$. Over the binary alphabet, there are only two involutory antimorphisms. First, the \textit{reversal map} $R$ given by $R(0)= 0$ and $R(1) = 1$. Second, the \textit{exchange antimorphism} satisfying $E(0) =1$ and $E(1) = 0$. We will use the following notation: $\overbar{0} = 1$, $\overbar{1} = 0$, $\overbar{R} = E$, and $\overbar{E} = R$. Over the ternary alphabet, there are exactly four involutory antimorphisms, denoted by $E_0$, $E_1$, $E_2$, and $R$:
\begin{itemize}
	\item $E_0(0) = 0$, $E_0(1) = 2$, and $E_0(2) = 1$,
	\item $E_1(0) = 2$, $E_1(1) = 1$, and $E_1(2) = 0$,
	\item $E_2(0) = 1$, $E_2(1) = 0$, and $E_2(2) = 2$,
	\item $R(0) = 0$, $R(1) = 1$, and $R(2) = 2$.
\end{itemize}

\begin{observation}
\label{oEiEjEk}
$E_iE_jE_k = E_j$ for $i,j,k \in \{0,1,2\}$ pairwise different.
\end{observation}
\begin{proof}
\begin{align*}
E_iE_jE_k(i) = E_i E_j (j) = E_i (j) = k = E_j(i), \\
E_iE_jE_k(j) = E_i E_j (i) = E_i (k) = j = E_j(j), \\
E_iE_jE_k(k) = E_i E_j (k) = E_i (i) = i = E_j(k).
\end{align*}
\end{proof}

\begin{definition}\label{def:image}
Let $w$ be a ternary word. Then any element of the set $\{\vartheta(w)| \vartheta \in \{E_0, E_1$, $E_2, R\} \}$ is called an \textit{image of $w$}.
\end{definition}

Let $\T$ be an involutory antimorphism. A finite word $w$ is a \textit{$\T$-palindrome} if $w = \T(w)$. For example, over the binary alphabet the word $00100$ is an $R$-palindrome (or just \textit{palindrome}) and the word $01001101$ is an $E$-palindrome. Over the ternary alphabet, $0112$ is an $E_1$-palindrome and $12010120$ is an $E_2$-palindrome. If we do not need to specify which antimorphism is used, we can say $w$ is a \textit{pseudopalindrome}.

\begin{observation}
\label{o_overpal}
Let $w$ be a ternary word and $\tilde{w}$ be its image, i.e., $\tilde{w} = \vartheta (w)$ for $\vartheta \in \{R, E_0, E_1,E_2\}$. Let $p$ be a~suffix of $w$. Then the word $v = w p^{-1} \tilde{w}$, where $|p| \geq 0$, is a pseudopalindrome. Moreover, it is a $\vartheta$-palindrome.
\end{observation}
\begin{proof}
First, let $|p| = 0$. Then $\vartheta (w \tilde{w}) = \vartheta (w \vartheta (w)) = \vartheta \vartheta (w) \vartheta (w) = w \vartheta(w) = w \tilde{w}$.

Now, let $|p| > 0$. If $v = w p^{-1} \tilde{w}$, then we know that $w = up$ and that $\vartheta (w) = p\vartheta (u)$. At the same time, $\vartheta(w)=\vartheta(up)=\vartheta(p)\vartheta(u)$. We obtain that $p = \vartheta (p)$. Hence, $\vartheta (v) = \vartheta (up\vartheta (u)) = \vartheta \vartheta (u) \vartheta(p) \vartheta (u) = u p \vartheta(u) = v$.
\end{proof}

\begin{observation}
\label{o_natureofpalonpal}
Let $w = \vartheta_1 (w)$, $\vartheta_1, \vartheta_2 \in \{R, E_0, E_1,E_2\}$, $\T_1 \neq \T_2$. Then $\T_2 (w)$ is a pseudopalindrome. Moreover :
\begin{itemize}
	\item If $\T_1 = R$, then $\T_2 (w)$ is an $R$-palindrome.
	\item If $\T_2 = R$, then $R(w)$ is a $\T_1$-palindrome.
	\item If $\T_1 = E_i$ and $\T_2 = E_j$, then $\T_2(w)$ is an $E_k$-palindrome, where $\{i,j,k\}=\{0,1,2\}$.
\end{itemize}
\end{observation}
\begin{proof}
We have $\T_2 (w) = \T_2 (\T_1 (w))$.
\begin{itemize}
	\item If $\T_1 = R$, then $R(\T_2 (w)) = \T_2 (R(w)) = \T_2 (w)$.
	\item If $\T_2 = R$, then $\T_1 (\T_2 (w)) = \T_1 (R(w)) = R (\T_1(w)) = R(w) = \T_2 (w)$.
	\item If $\T_1 = E_i$ and $\T_2 = E_j$, then $E_k( \T_2 (w)) = E_kE_jE_i(w)= E_j(w) = \T_2 (w)$.
\end{itemize}
\end{proof}

\begin{definition}
The \textit{$\vartheta$-palindromic closure} $u^{\T}$ of some finite word $u$ is the shortest $\T$-palindrome having $u$ as prefix.
\end{definition}

\begin{remark}
The $\vartheta$-palindromic closure of some word $u$ can be found in the following way:
we find the longest $\vartheta$-palindromic suffix $p$ of $u$, then $u = vp$ and $u^{\T} = vp\T (v)$. For instance, we have $(01011)^R = 01011010$ (the longest $R$-palindromic suffix is $11$), $(01201)^{E_2} = 01201$ (the longest $E_2$-palindromic suffix is the whole word $01201$), $(01011)^{E_0} = 0101122020$ (the longest $E_0$-palindromic suffix is $\varepsilon$).
\end{remark}

%
%
%
%
%

\subsection{Generalized pseudostandard words}
	Generalized pseudostandard words were first introduced in the paper \citep{bm} as a generalization of words obtained by pseudopalindromic closure with only one antimorphism.
	
	\begin{definition}
	\label{Dpseudo}
		Let $\mathcal{A}$  be an alphabet and $G$ be the set of all involutory antimorphisms on $\A ^*$. Let $\Delta = \delta_1 \delta_2 \ldots$ and $\Theta = \vartheta_1 \vartheta_2 \ldots$, where $\delta_i \in \A$ and $\vartheta_i \in G$ for all $i \in \N$. The infinite \textit{generalized pseudostandard word} $\mbu \DT$ is the word whose prefixes $w_n$ are obtained from the recurrence relation
		\begin{equation}
		\label{Dgps}
		\begin{aligned}
			w_{n+1} &= (w_n \delta_{n+1})^{\vartheta_{n+1}},\\
			w_0 &= \varepsilon.
		\end{aligned}
		\end{equation}
		The sequence $\DT$ is called the \textit{directive bi-sequence} of the word $\mbu\DT$.
	\end{definition}

	\begin{example}
	\label{E_DT}
				$\Delta = 01021\ldots$, $\Theta = RE_1E_1E_2R\ldots$
				\begin{align*}
				w_0 =\ & \varepsilon \\
				w_1 =\ & (0)^{R}=0\\
				w_2 =\ & (01)^{E_1}= 012\\
				w_3 =\ & (0120)^{E_1} = 012012\\
				w_4 =\ & (0120122)^{E_2}= 012012201201\\
				w_5 =\ & (0120122012011)^{R} = 012012201201102102210210\\
				\vdots
				\end{align*}
	\end{example}
	In Example \ref{E_DT}, $w_n$ are pseudopalindromic prefixes of $\mbu \DT$. However, it is easily seen that the sequence $(w_n)$ does not contain all of them: for instance $01$, $0120$, $01201$ are pseudopalindromic prefixes and are not equal to any $w_n$. This was the reason to define normalized directive bi-sequences, which will be discussed in the next section.
	
\subsection{Normalization}
It can be easily seen that one pseudostandard word can be generated by different directive bi-sequences and that the sequence $(w_n)$ of a generalized pseudostandard word does not need to contain all pseudopalindromic prefixes of the generated word. For this reason, the notion of a normalized directive bi-sequence was introduced in \citep{bm}.

\begin{definition}
\label{Dnormalized}
		A finite or infinite directive bi-sequence $\DT$ of a pseudostandard word $\mbu\DT$ over an alphabet $\A$ is called \textit{normalized} if the sequence of prefixes $(w_n)$ defined in \eqref{Dgps} contains all pseudopalindromic prefixes of $\mbu\DT$.
\end{definition}

If a pseudopalindromic prefix is not contained in the sequence $(w_n)$, we say that this pseudopalindromic prefix was \textit{missed}. If a $\T$-palindromic prefix was missed between $w_n$ and $w_{n+1}$, then it has an image of $w_n$ (see Definition~\ref{def:image}) as its suffix. Images of $w_n$ contained in $w_{n+1}$ are very important while looking for pseudopalindromic prefixes because every pseudopalindromic prefix contains an image of $w_n$ as a suffix.

\begin{definition}
\label{Dwimages}
Let $w_n^{(0)} = w_n$ and let $i_1, \dots, i_k$, where $i_1 <\dots <i_k$, be all occurrences of images of $w_n$ in $w_{n+1}$ from Definition~\ref{Dpseudo}. Denote the images of $w_n$ starting in $i_1, \dots, i_k$ by $w_n^{(1)}, \ldots, w_n^{(k)}$ (clearly, $w_n^{(k)}$ is a suffix of $w_{n+1}$). Furthermore, denote $w_n^{(j,m)}$ the factor stating in $i_j$ and ending in $i_m+|w_n|-1$, i.e., the factor $w_n^{(j,m)}$ has $w_n^{(j)}$ as prefix and $w_n^{(m)}$ as suffix.
\end{definition}

The authors of \citep{bm} showed that every binary directive bi-sequence can be normalized, i.e, a unique directive bi-sequence can be found such that it generates the same word and the corresponding sequence $(w_n)$ contains all $R$- and $E$-palindromic prefixes. Their result is summarized in the next theorem:

\begin{theorem}
\label{Normalizace01}
Let $\DT$ be a directive bi-sequence of a binary generalized pseudostandard word. Then there exists exactly one normalized directive bi-sequence $(\tilde{\Delta}, \tilde{\Theta})$ such that $\mbu\DT = \mbu (\tilde{\Delta}, \tilde{\Theta})$.
Moreover, in order to get the normalized bi-sequence $(\tilde{\Delta}, \tilde{\Theta})$ from $\DT$, it is sufficient to replace the prefix (if it is of the following form):

\begin{itemize}
\item $(a\ba, RR) \rightarrow (a \ba a, RER) $,
\item $(a^i, R^{i-1}E) \rightarrow (a^i \ba , R^i E)$ for $i \geq 1,$
\item $(a^i \ba \ba, R^i EE) \rightarrow (a^i \ba \ba a, R^iERE)$ for $i \geq 1$,
\end{itemize}
and then, to replace from left to right any factor
\begin{itemize}
\item $(ab \overbar{b}, \vartheta \overbar{\vartheta} \overbar{\vartheta})$ with $(ab\overbar{b}  b, \vartheta \overbar{\vartheta} \vartheta \overbar{\vartheta})$,
\end{itemize}
where $a,b \in \{0,1\}$ and $\vartheta \in \{E, R\}$.
\end{theorem}

Theorem \ref{Normalizace01} shows an easy-to-use algorithm.
A natural question follows. Does there exist an algorithm that normalizes every directive bi-sequence over a ternary alphabet?

\noindent The next chapter responds affirmatively to this question and presents a similar (but more complex) algorithm.

\section{Normalization over a ternary alphabet} \label{Normalization}

\subsection{The number of missed pseudopalindromic prefixes}

The aim of this section is to prove that, over a ternary alphabet, at most two pseudopalindromic prefixes may be missed between $w_n$ and $w_{n+1}$ from Definition \ref{Dpseudo}.

\begin{assumption}\label{assumption}
Let $w_n^{(1)}, \dots, w_n^{(k)}$, $k \geq 2$, be images of $w_n$ in $w_{n+1}$ from Definition~\ref{Dwimages} such that there exists $j \in \{1,\dots, k-1\}$ satisfying $w_n^{(j)}$ overlaps with both $w_n^{(0)}$ and $w_n^{(k)}$. (It is obvious that at least one pseudopalindromic prefix was missed between $w_n$ and $w_{n+1}$ in this case.)
\end{assumption}

\begin{lemma}
\label{Lstairs}
Let Assumption~\ref{assumption} hold. Then the length of the overlap of $w_n^{(i)}$ and $w_n^{(i+1)}$ is the same for all valid $i$.
\end{lemma}
\begin{proof}
Using Assumption~\ref{assumption}, we have that three consecutive images $w_n^{(i)}$, $w_n^{(i+1)}$, and $w_n^{(i+2)}$ overlap pairwise for every possible $i$. Suppose there exists a triplet $w_n^{(i)}$, $w_n^{(i+1)}$, and $w_n^{(i+2)}$ such that $w_n^{(i+1)}$ is not a central factor of $w_n^{(i, i+2)}$. Since $w_n^{(i, i+2)}$ is a pseudopalindrome by Observation \ref{o_overpal}, there exists another image of $w_n$ that is not included in the sequence of $w_n^{(i)}$, which is a contradiction.

\end{proof}

\begin{theorem}
\label{Ttwopseudo}
At most two pseudopalindromic prefixes may be missed between the prefixes $w_n$ and $w_{n+1}$ of $\mbu \DT$ from Definition \ref{Dpseudo}.

\end{theorem}
\begin{proof}
\begin{enumerate}
\item  First, suppose that Assumption~\ref{assumption} holds. We will now consider the possible palindromic nature of $w_n^{(0)}$ and $w_n^{(0,1)}$, see Figure \ref{Oschodynek} for a better understanding:

	\begin{figure}[ht!]
		\begin{center}
			\includegraphics[scale=1.1]{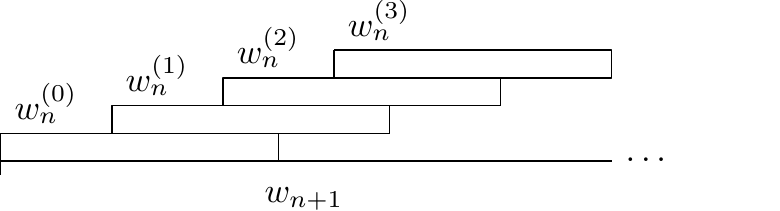}
		\end{center}
		\caption{Overlaps of $w_n$ and its images in $w_{n+1}$.}
		\label{Oschodynek}
	\end{figure}
	
\begin{itemize}

	\item $w_n^{(0)}$ is an $R$-palindrome and $w_n^{(0,1)}$ is an $R$-palindrome:
	
In order to construct $w_n^{(0,2)}$, we seek the longest $\T$-palindromic suffix of $w_n^{(0,1)}$, which is $R(w_n^{(0)})$. Thus $\T = R$ and $w_n^{(0,2)}$ is an $R$-palindrome. Analogously, we can deduce that all $w_n^{(0,j)}$ are $R$-palindromes and hence $w_{n+1}$ is also an $R$-palindrome. But this means that no palindromic prefix was missed between $w_n$ and $w_{n+1}$ by the construction of $w_{n+1}$ as the $R$-palindromic closure of $w_n$.
	
	\item $w_n^{(0)}$ is an $E_i$-palindrome and $w_n^{(0,1)}$ is an $E_i$-palindrome:
	
	Using similar arguments as in the case above, we deduce that no pseudopalindromic prefix was missed between $w_n$ and $w_{n+1}$.
	
	\item $w_n^{(0)}$ is an $R$-palindrome and $w_n^{(0,1)}$ is an $E_i$-palindrome:
	
	Now, in order to construct $w_n^{(0,2)}$, we look for the longest $\T$-palindromic suffix of $w_n^{(0,1)}$, which is $E_i(w_n^{(0)})$. It is an $R$-palindrome, thus $w_n^{(0,2)}$ is an $R$-palindrome, too. Similarly, in order to obtain $w_n^{(0,3)}$, the longest $\T$-palindromic suffix of $w_n^{(0,2)}$ is $R(w_n^{(0,1)})$, which is an $E_i$-palindrome and so is $w_n^{(0,3)}$. We get that $w_n^{(0,2)}$ is an $R$-palindrome, $w_n^{(0,3)}$ is an $E_i$-palindrome, $w_n^{(0,4)}$ an $R$-palindrome etc. There are two possibilities for $w_{n+1}$ since it is obtained by a pseudopalindromic closure: It is either an $R$-palindrome and then $w_{n+1}=w_n^{(0,2)}$ and one $E_i$-palindromic prefix was missed. Or, it is an $E_i$-palindrome and then $w_{n+1}=w_n^{(0,1)}$ and no pseudopalindromic prefix was missed.
	
	\item $w_n^{(0)}$ is an $E_i$-palindrome and $w_n^{(0,1)}$ is an $R$-palindrome:
	
	Similarly as in the case above, at most one pseudopalindromic prefix may be missed.
	
	\item $w_n^{(0)}$ is an $E_i$-palindrome and $w_n^{(0,1)}$ is an $E_j$-palindrome:
	
	When we construct $w_n^{(0,2)}$, the longest $\T$-palindromic suffix of $w_n^{(0,1)}$ is $E_j(w_n^{(0)})$, by Lemma \ref{o_natureofpalonpal}, it is an $E_k$-palindrome. Hence, $w_n^{(0,2)}$ is an $E_k$-palindrome. Following the steps, we deduce that $w_n^{(0)}$, $w_n^{(0,1)}$, $w_n^{(0,2)}, \ldots$ are successively  $E_i$-, $E_j$-, $E_k$-, $E_i$-, $E_j$-, $E_k$-, $\ldots$ palindromes.
	
	Since $w_{n+1}$ is constructed using a pseudopalindromic closure, it follows that at most two pseudopalindromic prefixes were missed between $w_n$ and $w_{n+1}$.
	
\end{itemize}

\item Now, we will address the situation where Assumption~\ref{assumption} is not satisfied, i.e., there is no image of $w_n$ that overlaps with the prefix occurrence of $w_n$ and the suffix occurrence of an image of $w_n$. This can happen only if $2|w_n| + 1 \leq |w_{n+1}| \leq 2|w_n| + 2 $.
\begin{itemize}
	\item $|w_{n+1}| = 2|w_n| + 1$:
	In this case, it is easy to see that at most two pseudopalindromic prefixes were missed because there are only two ways to place the images of $w_n$ inside $w_{n+1}$ so that Assumption~\ref{assumption} is not satisfied.

	\item $|w_{n+1}| = 2|w_n| + 2$:
	Here, there are four ways to place the images of $w_n$ inside $w_{n+1}$ so that Assumption~\ref{assumption} is not satisfied. Suppose that we place three images of $w_n$ inside $w_{n+1}$, for example as in Figure \ref{Ospecialni}. By Observation~\ref{o_overpal}, $w_n^{(0,3)}$ is a pseudopalindrome. It is thus easily seen that $w_{n+1}$ contains another image of $w_n$ that satisfies Assumption~\ref{assumption}, which is a contradiction. The other possible cases can be excluded in a similar way.
	
	\begin{figure}[ht!]
		\begin{center}
			\includegraphics[scale=1.1]{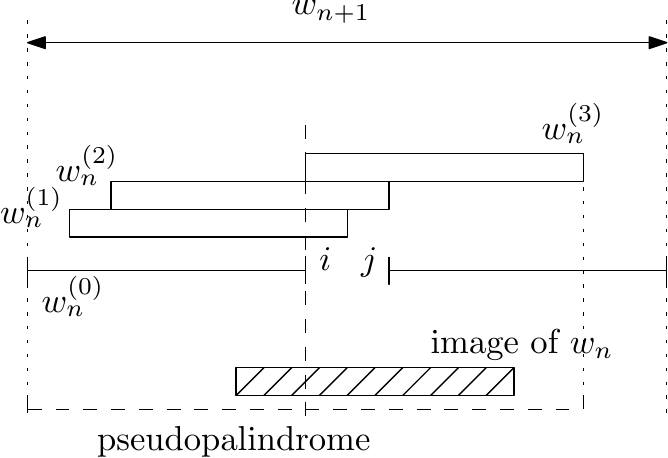}
		\end{center}
		\caption{Illustration of the contradiction for the case $|w_{n+1}| = 2|w_n| + 2$.}
		\label{Ospecialni}
	\end{figure}
	
\end{itemize}

\end{enumerate}

\end{proof}

\begin{corollary}
\label{CstairsType}
Let Assumption~\ref{assumption} be satisfied. Then we can deduce from the proof of Theorem \ref{Ttwopseudo} the following statements:
\begin{itemize}
\item If exactly one pseudopalindromic prefix $w_n^{(0,1)}$ was missed between $w_n$ and $w_{n+1}$, then $w_n^{(0)} = w_n$, $w_n^{(0,1)}$, and $w_n^{(0,2)} = w_{n+1}$ are successively either an $R$-, $E_i$-, and $R$-palindrome or an $E_i$-, $R$-, and $E_i$-palindrome or an $E_i$-, $E_j$-, and $E_k$-palindrome.
\item If exactly two pseudopalindromic prefixes $w_n^{(0,1)}$ and $w_n^{(0,2)}$ were missed between $w_n$ and $w_{n+1}$, then $w_n^{(0)}$, $w_n^{(0,1)}$, $w_n^{(0,2)}$, and $w_n^{(0,3)}= w_{n+1}$ are successively an $E_i$-, $E_j$-, $E_k$-, and $E_i$-palindrome.
\end{itemize}

\end{corollary}

\subsection{Special cases}
In the proof of Theorem \ref{Ttwopseudo}, we set apart the instances where Assumption~\ref{assumption} was not satisfied. In this section, we will investigate separately those cases and show that they lead only to special cases of infinite words.
In the first place, we will state three useful lemmas discussing cases where $w$ is a pseudopalindrome and $wa$ or $wab$ are also pseudopalindromes for $a,b \in \A$. This kind of pseudopalindromes appears to be significant for examining the words where  Assumption~\ref{assumption} is not satisfied.

\begin{lemma} \label{tvarjednonavic}
Let $\mathcal{A}$ be a finite alphabet, $n \in \mathbb{N}_0$, $a_{n+1}  \in \mathcal{A}$, and $\vartheta_1$, $\vartheta_2$ be two involutory antimorphisms over $\mathcal{A}$. Let $w = \vartheta_1(w)=a_1\dots a_n$. Then $wa_{n+1}=\vartheta_2(wa_{n+1})$ if, and only if, $wa_{n+1}=a_1\vartheta_2\vartheta_1(a_1)\dots(\vartheta_2\vartheta_1)^n(a_1)$ and $a_{n+1}=\vartheta_2(a_1).$
\end{lemma}

\begin{proof}
It is a~direct consequence of Lemma 17 from~\cite{bm}.
%
\end{proof}

\begin{lemma} \label{Ljednonavic}
Let $\mathcal{A}=\{0,1,2\}$, $n \in \mathbb{N}_0$, $a_{n+1} \in \mathcal{A}$, and $\vartheta_1, \vartheta_2 \in \{E_0,E_1,E_2,R\}$. Furthermore, let $w=\vartheta_1(w)$. If $wa_{n+1} = \vartheta_2(wa_{n+1})$, then there exist $i,j,k \in \mathcal{A}$ pairwise different such that $wa_{n+1}$ is the prefix of length $n+1$ of one of the following infinite words:
\begin{itemize}
	\item $i^{\omega}$,
	\item $(ij)^{\omega}$,
	\item $(ijk)^{\omega}$.
\end{itemize}
\end{lemma}
\begin{proof}
From Lemma~\ref{tvarjednonavic} we know that $wa_{n+1}=\vartheta_2(wa_{n+1})$, if, and only if, \\ $wa_{n+1}=a_1\vartheta_2\vartheta_1(a_1)\dots(\vartheta_2\vartheta_1)^n(a_1)$ and $a_{n+1}=\vartheta_2(a_1).$ We will address in the sequel the different possible cases of the antimorphisms $\vartheta_1$ and $\vartheta_2$ and the letter $a_1$ denoted by $i$:
\begin{itemize}
\item
$\vartheta_1=\vartheta_2=R$, then $\vartheta_2\vartheta_1=I$, hence $wa_{n+1}=i^{n+1}$,
\item
$\vartheta_1=\vartheta_2=E_i,$ then $\vartheta_2\vartheta_1=I$, hence $wa_{n+1}=i^{n+1}$,
\item
$\vartheta_1=R, \vartheta_2=E_i$ (or the other way around), then $\vartheta_2\vartheta_1=E_iR$, thus $wa_{n+1}=i^{n+1}$,
\item
$\vartheta_1=R, \vartheta_2=E_k$, then $n$ is odd and $\vartheta_2\vartheta_1=E_kR$, thus $wa_{n+1}=(ij)^{\frac{n+1}{2}}$,
\item
$\vartheta_1=E_k, \vartheta_2=R$, then $n$ is even and $\vartheta_2\vartheta_1=RE_k$, thus $wa_{n+1}=(ij)^{\frac{n}{2}}i$,
\item
$\vartheta_1=E_k, \vartheta_2=E_j$, then $n\equiv2$\;mod\;$3$ and $\vartheta_2\vartheta_1=E_jE_k$, thus $wa_{n+1}=(ijk)^{\frac{n+1}{3}}$,
\item
$\vartheta_1=E_j, \vartheta_2=E_i$, then $n\equiv0$\;mod\;$3$ and $\vartheta_2\vartheta_1=E_iE_j$, thus $wa_{n+1}=(ijk)^{\frac{n}{3}}i$,
\item
$\vartheta_1=E_i, \vartheta_2=E_k$, then $n\equiv1$\;mod\;$3$ and $\vartheta_2\vartheta_1=E_kE_i$, thus $wa_{n+1}=(ijk)^{\frac{n-1}{3}}ij$.
\end{itemize}
\end{proof}

\begin{lemma} \label{Ldvenavic}
Let $w_n$ and $w_{n+1}$ be the prefixes of $\mbu \DT$ from Definition \ref{Dpseudo}, where $|w_{n+1}| = |w_n| + 2$. If Assumption~\ref{assumption} is not satisfied, then $w_{n+1}$ is a prefix of one of the following infinite words:
\begin{itemize}
\item
$(ij)^{\omega}$,
\item
$(ijji)^{\omega}$,
\item
$(ijk)^{\omega}$,
\item
$(ijik)^{\omega}$,
\item
$(ijjkki)^{\omega}$,
\item
$(ijkj)^{\omega}$,
\item
$(iijj)^{\omega}$,
\item
$(iijjkk)^{\omega}$,
\end{itemize}
where $i,j,k \in \{0,1,2\}$ are pairwise distinct letters.
\end{lemma}
\begin{proof}
If $w_n = i$, then $w_{n+1} \in \{iji, ijk\}$, and if $|w_n| = 2$, then $w_{n+1} \in \{iijj, ijij, ijji$, $ijjk, ijki\}$. Now, we can suppose that $|w_n| > 2$. Let $w_n^{(p)}$ denote the prefix occurrence of $w_n$ in $w_{n+1}$ and $w_n^{(s)}$ denote the suffix occurrence of an image of $w_n$ in $w_{n+1}$.

Suppose first that $w_{n+1}$ end in two different letters. Without loss of generality, let those two letters be $01$. Since $w_n^{(s)}$ has $01$ as a suffix, then $w_n^{(p)}$ has two different letters, say $ij$, as a suffix, too. Suppose that $w_n^{(p)}$ is a $\vartheta_1$-palindrome and that $w_n^{(s)} = \vartheta_2 (w_n^{(p)})$ (i.e., $w_n=\vartheta_1(w_n)$ and $w_{n+1}=\vartheta_2(w_{n+1})$). The factor $w_n^{(p)}$ has then $\vartheta_2 (01)$ as prefix, thus $\vartheta_1 \vartheta_2 (01)$ as suffix. It implies that $w_{n+1}$ is a~suffix of $\dots(\vartheta_1 \vartheta_2)^2 (01)(\vartheta_1 \vartheta_2) (01)01$ etc. Let us examine the different possible cases:

\begin{itemize}
	\item $ij = 01$, then $w_n01$ is a prefix of $(01)^{\omega}$ or $(10)^{\omega}$,
	\item $ij = 10$, then $w_n01$ is a prefix of $(0110)^{\omega}$ or $(1001)^{\omega}$,
	\item $ij = 12$, then $w_n01$ is a prefix of $(201)^{\omega}$, $(120)^{\omega}$, or $(012)^{\omega}$,
	\item $ij = 21$, then $w_n01$ is a prefix of $(1210)^{\omega}$ or $(1012)^{\omega}$, \\ here, the words $(2101)^{\omega}$ and $(0121)^{\omega}$ had been considered, but their prefixes ending in $01$ are not pseudopalindromes,
	\item $ij = 20$, then $w_n01$ is a prefix of $(122001)^{\omega}$, $(200112)^{\omega}$, or $(011220)^{\omega}$,
	\item $ij = 02$, then $w_n01$ is a prefix of $(1020)^{\omega}$ or $(2010)^{\omega}$, \\
other words could have also been considered, but their prefixes ending in $01$ are again not pseudopalindromes.
\end{itemize}

Now, we have to  address the situation where the suffix of $w_{n+1}$ of length two is equal to $ii$. Once again, the suffix of $w_n^{(p)}$ will be $\vartheta_1\vartheta_2(ii)$. The word $w_{n+1}$ can now only be a suffix of $(jjii)^{\omega}$ or $(kkjjii)^{\omega}$.
\end{proof}

The aim of the following observations is to find all possible pseudopalindromic prefixes $w_n$ and $w_{n+1}$ such that a pseudopalindromic prefix was missed between them and Assumption~\ref{assumption} is not satisfied.

\begin{proposition}
\label{Ospecwijw}
Let $w_n$ and $w_{n+1}$ be the prefixes of $\mbu \DT$ from Definition \ref{Dpseudo} and $|w_{n+1}| = 2|w_n| + 2$. Suppose Assumption~\ref{assumption} does not hold. Furthermore, suppose the overlap of $w_n^{(0)}$ and $w_n^{(1)}$ is either of length $|w_n| - 1$ or of length $|w_n| - 2$. Then $w_{n+1}$, $w_n$, and the missed pseudopalindromic prefix(es) are of the form:
\begin{table}[ht!]
\catcode`\-=12
\centering
\bgroup
\def\arraystretch{1.2}

\begin{tabular}{|l||l|l|}
\hline
$w_{n+1}$              & $w_n$       & Missed pseudopalindromic prefix(es) \\ \hline \hline
$i^{l}j^{l}$       		    & $i^{l-1}$        	& $i^{l}$                      	\\ \hline
$(ij)^l(ki)^l$         		& $(ij)^{l-1} i$   	& for $l \geq 2$: $(ij)^l$, for $l = 1$: $ij$, $ijk$ \\ \hline
$(ij)^lik(jk)^l$       		& $(ij)^l$         	& $(ij)^l i$                       \\ \hline
$ijjkkiijjk$        		& $ijjk$    		& $ijjkki$, $ijjkkiij$             \\ \hline
$ijjkki$     				& $ij$   			& $ijjk$               				\\ \hline
$ijjiij$ 					& $ij$ 				& $ijji$             				\\ \hline
$(ijkj)^l(ikij)^l$    		& $(ijkj)^{l-1} ijk$ 	& $(ijkj)^l i$               		\\ \hline
$(ijkj)^lijki(kjki)^l$    	& $(ijkj)^l i$   	& $(ijkj)^l ijk$                		\\ \hline
\end{tabular}

\egroup

\end{table}

\end{proposition}

\begin{proof}
Since $|w_{n+1}| = 2|w_n|+2$, it is easily seen that $w_{n+1}$ is of the form $w_n ij E_k(w_n)$, where $i, j, k$ are pairwise different letters. Without loss of generality, suppose that $w_{n+1} = w_n01E_2(w_n)$. Two cases are possible:
\begin{enumerate}
\item
The length of the overlap of $w_n^{(0)}$ and $w_n^{(1)}$ is equal to $|w_n|-1$.

If $|w_n| = 1$, the only possibility is $w_{n+1} \in \{0011, 2012\}$. Consider $|w_n|\geq 2$. By Observation \ref{o_overpal} and Lemma~\ref{Ljednonavic}, $w_n0$ is a prefix of $i^{\omega}$, $(ij)^{\omega}$, or $(ijk)^{\omega}$, i.e., either $w_n0=0^{l}$ and $w_{n+1}=0^{l}1^{l}$, or $w_n0$ is a prefix of $(i0)^{\omega}$, $(0i)^{\omega}$, or $w_n0$ is a prefix of $(ij0)^{\omega}$, $(i0j)^{\omega}$, or $(0ij)^{\omega}$. In the case where $w_n0$ is a prefix of $(ij0)^{\omega}$, $(i0j)^{\omega}$, or $(0ij)^{\omega}$, the longest $E_2$-palindromic prefix of $w_n0$ is clearly not an empty word, thus this case cannot happen. We will now address the remaining possibility: $w_n0$ is a prefix of $(i0)^{\omega}$ or $(0i)^{\omega}$. The case $i = 1$ cannot happen for the same reason as in the previous cases. Hence, only the case $i = 2$ remains possible. It leads to two possible forms of $w_{n+1}$: $w_n01E_2(w_n) = (20)^l(12)^l$ and $w_n01E_2(w_n) = (02)^l01(21)^l$.

Overall, we get $w_{n+1} = 0^{l}1^{l}$, $w_{n+1} = (20)^{l}(12)^{l}$, or $w_{n+1}=(02)^l01(21)^l$.
\item
The length of the overlap of $w_n^{(0)}$ and $w_n^{(1)}$ is equal to $|w_n|-2$.

If $|w_n| = 2$, then $w_{n+1} \in \{100110, 200112\}.$ Consider $|w_n|>2$. The word $w_n01$ has an image of $w_n$ as a suffix. Let $ij$ denote the factor preceding $01$. By Lemma \ref{Ldvenavic}, the following cases can happen:

\begin{enumerate}
\item If $ij = 01$, then $w_n01$ is a prefix of $(01)^{\omega}$ or $(10)^{\omega}$,
\item If $ij = 10$, then $w_n01$ is a prefix of $(0110)^{\omega}$ or $(1001)^{\omega}$,
\item If $ij = 12$, then $w_n01$ is a prefix of $(201)^{\omega}$, $(120)^{\omega}$, or $(012)^{\omega}$,
\item If $ij = 21$, then $w_n01$ is a prefix of $(1210)^{\omega}$ or $(1012)^{\omega}$,
\item If $ij = 20$, then $w_n01$ is a prefix of $(122001)^{\omega}$, $(200112)^{\omega}$, or $(011220)^{\omega}$,
\item If $ij = 02$, then $w_n01$ is a prefix of $(1020)^{\omega}$ or $(2010)^{\omega}$.
\end{enumerate}

The cases (a), (c), (d) are not possible because the longest $E_2$-palindromic suffix of $w_n0$ is not the empty word (in the case (a) and (d), $10$ is an $E_2$-palindromic suffix of $w_n0$, in the case (c), $120$ is an $E_2$-palindromic suffix of $w_n0$).

The case (b) happens only for $w_n01 = 1001$ ($w_{n+1} = 100110$). Otherwise, $1100$ is an $E_2$-palindromic suffix of the word $w_n0$.

Similarly, the case (e) occurs for $w_n01 \in \{122001, 2001\}$, it follows that $w_{n+1} \in \{1220011220, 200112\}$. Otherwise, $112200$ is an $E_2$-palindromic suffix of $w_n0$.

The case (f) can happen and leads to $w_{n+1} \in \{(1020)^l(1210)^l, (2010)^l2012(1012)^l\}$.
\end{enumerate}
\end{proof}

\begin{proposition}
\label{Ospecwiw}
Let $w_n$ and $w_{n+1}$ be the prefixes of $\mbu \DT$ from Definition \ref{Dpseudo} and $|w_{n+1}| = 2|w_n| + 1$. Suppose Assumption~\ref{assumption} does not hold. Furthermore, suppose the overlap of $w_n^{(0)}$ and $w_n^{(1)}$ is of length $|w_n| - 1$. Then $w_{n+1}$, $w_n$, and the missed pseudopalindromic prefix(es) are of the form:

\begin{table}[ht!]
\catcode`\-=12
\centering
\bgroup
\def\arraystretch{1.2}

\begin{tabular}{|l||l|l|}
\hline
$w_{n+1}$              & $w_n$       & Missed pseudopalindromic prefix(es) \\ \hline \hline
$iji$                  & $i$         & $ij$                       \\ \hline
$ijk$                  & $i$         & $ij$                       \\ \hline
$(ij)^li(ki)^l$        & $(ij)^l$    & $(ij)^li$                  \\ \hline
$(ij)^l ijk(jk)^l$     & $(ij)^li$   & $(ij)^{l+1}$               \\ \hline
$ijkij$                & $ij$        & $ijk, ijki$                \\ \hline
$(ijk)^{l-1} ijkji(kji)^{l-1}$ & $(ijk)^{l-1}ij$ & $(ijk)^l$              \\ \hline
$(ijk)^l iji(kji)^l$    & $(ijk)^l i$ & $(ijk)^l ij$               \\ \hline
$(ijk)^l i(kji)^l$     & $(ijk)^l$   & $(ijk)^l i$                \\ \hline
\end{tabular}

\egroup

\end{table}

\end{proposition}

\begin{proof}
Since $|w_{n+1}|=2|w_n|+1$, it directly follows that $w_{n+1}$ is either of the form $w_{n+1}=w_n j R(w_n)$ or $w_{n+1} = w_n j E_j(w_n)$ for some $j \in \{0, 1, 2\}$ because $j$ is a central factor of $w_{n+1}$. If $|w_n| = 1$, then $w_{n+1} \in \{iji, ijk\}$, where $i, j, k$ are pairwise different letters. Now, we can suppose that $|w_n| \geq 2$.

Without loss of generality, assume that $w_{n+1} = w_n0R(w_n)$, or $w_{n+1} = w_n0E_0(w_n)$.
By Observation \ref{o_overpal} and Lemma~\ref{Ljednonavic}, $w_n0$ is a prefix of $i^{\omega}$, $(ij)^{\omega}$, or $(ijk)^{\omega}$, i.e., $w_n0$ is a prefix of $0^{\omega}$, $(i0)^{\omega}$, $(0i)^{\omega}$, $(ij0)^{\omega}$, $(i0j)^{\omega}$, or $(0ij)^{\omega}$.

The first case cannot happen because the longest $E_0$- or $R$-palindromic suffix of $w_n0$ is $00$.

If $w_n0$ is a prefix of $(i0)^{\omega}$, or $(0i)^{\omega}$, then $w_{n+1} \in \{(0i)^l0(j0)^l), (i0)^li0j(0j)^l\}$ when we make an $E_0$-palindromic closure. The case of an $R$-palindromic closure cannot happen since the $R$-palindromic suffix is $0i0$.

If $w_n0$ is a prefix of $(ij0)^{\omega}$, $(i0j)^{\omega}$, or $(0ij)^{\omega}$, then, for the case of an $E_0$-palindromic closure, the only possibility is $w_n 0=ij0$, which leads to $w_{n+1}=ij0ij$ (for longer prefixes, $0ij0$ is an $E_0$-palindromic suffix of $w_n0$). In the $R$-palindromic closure case, we obtain $w_n0 \in \{(ij0)^{l-1}ij0ji(0ji)^{l-1}, (i0j)^li0i(j0i)^l, (0ij)^l0(ji0)^l\}$.
\end{proof}

\subsection{Missing one pseudopalindromic prefix}
Having solved special cases that can appear if Assumption~\ref{assumption} is not satisfied, we will restrict our attention to the cases where Assumption~\ref{assumption} holds. In this section, we will assume that we missed exactly one pseudopalindromic prefix and thus $w_n^{(0)}=w_n$, $w_n^{(0,1)}$, $w_n^{(0,2)} = w_{n+1}$ are the only pseudopalindromic prefixes between $w_n$ and $w_{n+1}$.

\begin{remark}
Through the whole section, $w_{n-1}^{(p)}$ denotes the prefix of length $|w_{n-1}|$ of $w_n^{(0)}$ (i.e., $w_{n-1}^{(p)}=w_{n-1}$), and $w_{n-1}^{(c)}$ denotes the prefix of length $|w_{n-1}|$ of $w_n^{(1)}$.
\end{remark}

\begin{lemma} \label{L_1_p}
Let $w_n = w_{n-1}p_1^{-1}\vartheta_1(w_{n-1})$ and $w_{n+1} = w_np_2^{-1}\vartheta_2(w_n)$, where $\vartheta_1, \vartheta_2 \in \{E_0, E_1, E_2, R\}$ and $p_1, p_2 \in \{0,1,2\}^*$. Furthermore, suppose that exactly one pseudopalindromic prefix was missed between $w_n$ and $w_{n+1}$ such that Assumption~\ref{assumption} holds and suppose that the prefix of length $n$ of the directive bi-sequence is normalized. Then $|p_1| = |p_2|$.
\end{lemma}
\begin{proof}
By Lemma \ref{Lstairs}, we know that $w_n^{(1)}$ is a central factor of $w_{n+1}$. We will proceed by contradiction:
\begin{enumerate}
	\item Assume $|p_1| > |p_2|$. For a better understanding see Figure \ref{Ooverlappings} (a). Then $w_{n-1}^{(p)}$ and $w_{n-1}^{(c)}$ overlap and the factor having $w_{n-1}^{(p)}$ as prefix and $w_{n-1}^{(c)}$ as suffix is a $\vartheta$-palindrome by Observation \ref{o_overpal}. Moreover, it is a $\vartheta$-palindromic prefix of $w_{n+1}$ longer than $w_n$ and shorter than $w_n^{(0,1)}$, which is a contradiction.
		\begin{figure}[ht!]
		\begin{center}
			\includegraphics[scale=1.1]{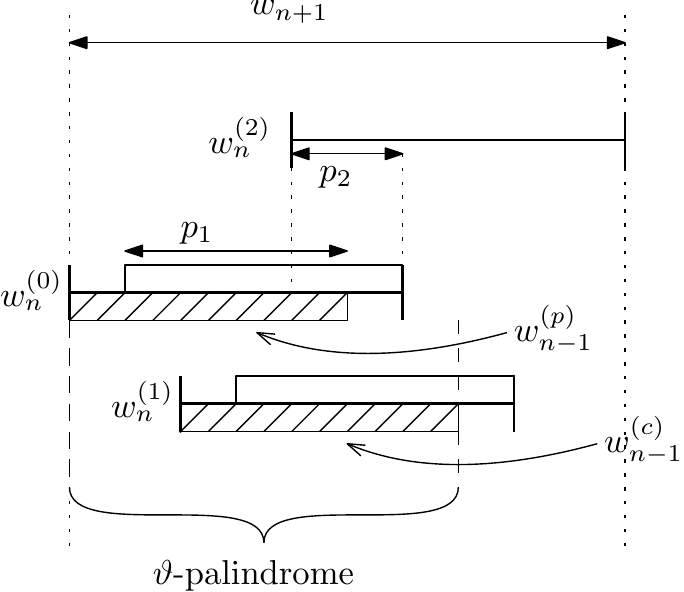}
            \includegraphics[scale=1.1]{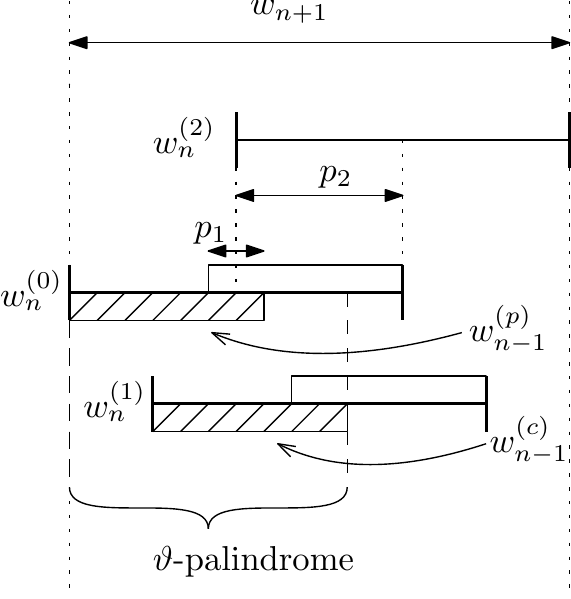}
		\end{center}
		\caption{(a) \ Illustration of the case $|p_1| > |p_2|$. \quad (b) \ Illustration of the case $|p_1| < |p_2|$.}
		\label{Ooverlappings}
	\end{figure}

	\item Assume $|p_1| < |p_2|$, see Figure \ref{Ooverlappings} (b). Again, $w_{n-1}^{(p)}$ and $w_{n-1}^{(c)}$ overlap and the factor having $w_{n-1}^{(p)}$ as prefix and $w_{n-1}^{(c)}$ as suffix is a $\vartheta$-palindrome for the same reason as above. However, in this case, the obtained $\vartheta$-palindromic prefix of $w_{n+1}$ is shorter than $w_n$, which is a contradiction with the fact that the prefix of length $n$ of the directive bi-sequence is normalized.

%
\end{enumerate}
\end{proof}
\clearpage

\begin{proposition} \label{L_1_1}
Suppose that exactly one pseudopalindromic prefix was missed between $w_n$ and $w_{n+1}$ such that Assumption~\ref{assumption} holds and suppose that the prefix of length $n$ of the directive bi-sequence is normalized. Then:
\begin{enumerate}
\item
If $w_n = w_{n-1}i\vartheta_1(w_{n-1})$ for some $i \in \{0,1,2\}$ and $\vartheta_1 \in \{E_0,E_1,E_2,R\}$, then $w_{n+1} = w_nj\vartheta_2(w_n)$ for some $j \in \{0,1,2\}$ and $\vartheta_2 \in \{E_0,E_1,E_2, R\}$.
\item
If $w_{n+1} = w_nj\vartheta_2(w_n)$ for some $j \in \{0,1,2\}$ and $\vartheta_2 \in \{E_0,E_1,E_2,R\}$, then either  $w_n = w_{n-1}i\vartheta_1(w_{n-1})$ for some $i \in \{0,1,2\}$ and $\vartheta_1 \in \{E_0,E_1,E_2,R\}$, or $w_{n+1}$ and $w_n$ are of one of the following forms for $l \geq 1$ ($w_n=w_{n-1}a^{-1}\vartheta_1(w_{n-1}$) for some $a \in \{0,1,2\}$):

\begin{table}[ht!]
\catcode`\-=12
\centering
\bgroup
\def\arraystretch{1.2}

\begin{tabular}{l|l||l|l|}
\cline{2-4}
\multicolumn{1}{c|}{}& $w_{n+1}$              		& $w_n$       		& Missed pseudopal. prefix\\
\hhline{~===}
a.&$(ijk)^lk^{-1}(kji)^lk(ijk)^lk^{-1}(kji)^l$   & $(ijk)^lk^{-1}(kji)^l$         & $(ijk)^lk^{-1}(kji)^lk(ijk)^lk^{-1}$                   \\ \cline{2-4}
b.&$(ijk)^li(kji)^lk(ijk)^li(kji)^l$                & $(ijk)^li(kji)^l$        & $(ijk)^li(kji)^lk(ijk)^l$                \\ \cline{2-4}
c.&\makecell[lr]{$(ijk)^{l-1} iji(kji)^{l-1} k(ijk)^{l-1} iji\ldots$\\$(kji)^{l-1}$}                  & $(ijk)^{l-1} iji(kji)^{l-1}$         & \makecell[lr]{$(ijk)^{l-1} iji(kji)^{l-1} k\ldots$ \\ $(ijk)^{l-1} i$  }                     \\ \cline{2-4}
d.&\makecell[lr]{$i(ji)^{l-1} jk(jk)^{l-1} j(ij)^{l-1} ij\ldots$\\$(kj)^{l-1} k$}       & $i(ji)^{l-1} jk(jk)^{l-1}$    & $i(ji)^{l-1} jk(jk)^{l-1} j(ij)^{l-1} i$                  \\ \cline{2-4}
e.&$(ij)^l i(ki)^l k(jk)^l j(ij)^l$    & $(ij)^l i(ki)^l$   & $
(ij)^l i(ki)^l k(jk)^l$               \\ \cline{2-4}
\end{tabular}

\egroup

\end{table}

\end{enumerate}
\end{proposition}
\begin{proof}
\begin{enumerate}
	\item If $w_{n-1}=\varepsilon$, then $w_n=i$ and $w_{n+1}\in\{iji, ijk\}$. In the sequel, assume $|w_{n-1}|\geq 1$. First, suppose that $w_n = w_{n-1}i\vartheta_1(w_{n-1})$ and $w_{n+1} = w_n p_2^{-1} \vartheta_2 (w_n)$ for some $p_2 \in \{0,1,2\}^*$, $|p_2| \geq 2.$ Then $w_{n-1}^{(p)}$ and $w_{n-1}^{(c)}$ overlap and, similarly as in the proof of Lemma \ref{L_1_p}, it is a contradiction with the fact that the prefix of length $n$ of the directive bi-sequence is normalized. Therefore, this case is not possible. Moreover, the length of $p_2$ has to be odd because $w_n^{(1)}$ is a central factor of $w_{n+1}$ and the length of $w_n$ is odd. Hence, $|p_2|$ is either $1$, or $w_{n+1} = w_nj\vartheta_2(w_n)$ for some $j \in \{0,1,2\}$.
	
	If $|p_2|=1$, then $w_{n-1}^{(c)}$ is a~prefix of the suffix $i\vartheta_1 (w_{n-1})$ of $w_n$. Since $w_n$ is a pseudopalindrome, then an image of $w_{n-1}$ is also a suffix of $w_{n-1}i$, see Figure \ref{OLSpec11}. This implies that $w_{n-1}i$ is a pseudopalindrome and this is a contradiction with the fact that the prefix of length $n$ of the directive bi-sequence is normalized. Thus, only the case $w_{n+1} = w_n j \vartheta_2(w_n)$ remains possible.
	\begin{figure}[ht!]
			\begin{center}
				\includegraphics[scale=1.1]{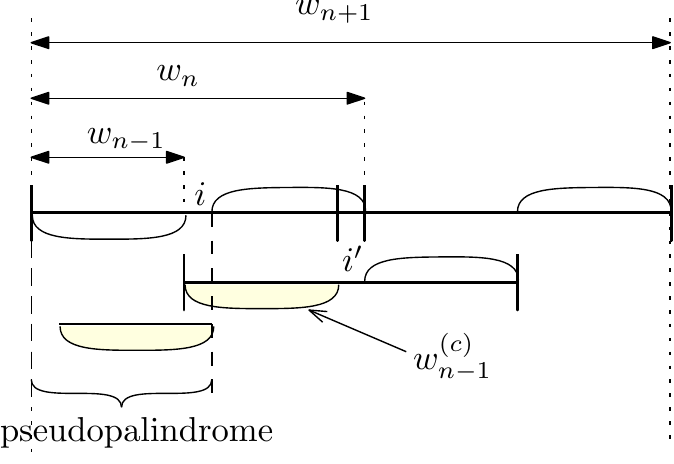}
			\end{center}
			\caption{Illustration of the case $|p_2|=1$.}
			\label{OLSpec11}
		\end{figure}
	
	\item We have now $w_{n+1} = w_nj\vartheta_2(w_n)$. We will consider all different possible lengths for $w_{n-1}$. First, suppose that $|w_{n-1}|= 0$. Then $w_n = i$ (since the length of $w_n$ is odd) and $w_{n+1} \in  \{iji, ijk\}$. Second, $|w_{n-1}| = 1$. Then $w_n \in \{iji, ijk\}$ because the length of $w_n$ is odd. But at the same time, $w_n \not \in \{iji, ijk\}$ because the prefix of length $n$ of the directive bi-sequence is normalized. From now on, suppose $|w_{n-1}| \geq 2$.
	
	Assume now that $w_n = w_{n-1}p_1^{-1} \vartheta_1 (w_{n-1})$  for some $p_1 \in \{0,1,2\}^*$, where $|p_1| \geq 2$. Then, as in the first part of the proof, the factors $w_{n-1}^{(p)}$ and $w_{n-1}^{(c)}$ overlap and two pseudopalindromic prefixes were missed between $w_n$ and $w_{n+1}$, which is not possible. Since the length of $p_1$ is odd, only the cases where $w_n = w_{n-1} i^{-1} \vartheta_1(w_{n-1})$ and $w_n = w_{n-1} i \vartheta_1(w_{n-1})$ remain.

	We will derive the remaining forms of $w_{n+1}$ from the case $w_n = w_{n-1} i^{-1} \vartheta_1(w_{n-1})$. Suppose that $w_n^{(0)}=w_n$, $w_n^{(0,1)}$, and $w_n^{(0,2)}=w_{n+1}$ are in order a $\vartheta_1$-, $\vartheta$-, and $\vartheta_2$-palindromes. Let $p$ denote the overlap of $w_n^{(0)}$ and $w_n^{(1)}$. See Figure \ref{OLSpec12} for a better understanding. It is easily seen that the length of $p$ is equal to $|w_{n-1}| - 1$ and that $p$ is a $\vartheta$-palindrome by Observation \ref{o_overpal}. Moreover, $pj$ is a pseudopalindrome, too ($pj$ is an image of $w_{n-1}$). Hence Lemma \ref{Ljednonavic} is applicable: $pj$ is a prefix of either $a^{\omega}$, $(ab)^{\omega}$, or $(abc)^{\omega}$ for some pairwise different $a,b,c \in \{0,1,2\}$.
	
			\begin{figure}[ht!]
			\begin{center}
				\includegraphics[scale=1.1]{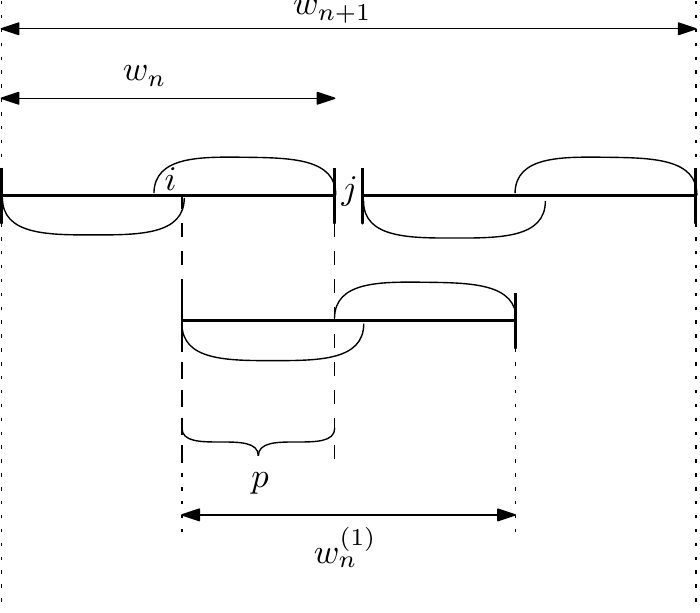}
			\end{center}
			\caption{Illustration of the overlap $p$ for $w_n = w_{n-1} i^{-1} \vartheta_1(w_{n-1})$.}
			\label{OLSpec12}
		\end{figure}
	
	Furthermore, $w_{n+1}$ clearly satisfies the following equation:
	
	\begin{equation}
	\label{eq:tipjtittp}
		w_{n+1} = \vartheta_1(p)ipj\vartheta_2 (p) \vartheta_2 (i) \vartheta_2 \vartheta_1 (p).
	\end{equation}

	By Corollary \ref{CstairsType}, we know that  $\vartheta_1$, $\vartheta$, and $\vartheta_2$ are either successively $R$, $E_m$, $R$, or $E_m$, $R$, $E_m$, or $E_m$, $E_r$, $E_s$. The different cases will be addressed in the sequel:

	\begin{itemize}
	\item $\vartheta_1 = R$, $\vartheta = E_m$, $\vartheta_2 = R$:
	
	Using Equation \eqref{eq:tipjtittp}, we obtain:
	\begin{equation}
	\label{eq:ripjrip}
		w_{n+1} = R(p)ipjR(p)ip,
	\end{equation}
	
where $i,j,m$ are either the same or pairwise different letters.

If $pj$ is a prefix of $a^{\omega}$, then $a=j$ and $pj$ is a palindromic suffix of $w_nj$, which is a contradiction with $w_{n+1} = w_njR(w_n)$ and $|w_n| \geq 2$.

If $pj$ is a prefix of $(ab)^{\omega}$, then, since $p$ is an $E_m$-palindrome, the letters $i,j,m$ are pairwise different and $p=(ji)^l$, which is also a contradiction with the fact that $w_{n+1} = w_{n}jR(w_n)$.

If $pj$ is a prefix of $(abc)^{\omega}$, then two cases can occur:
		\begin{enumerate}
			\item $i = j = m$, then $p = a(bja)^{l-1} b$ (it is easy to see that other cases are not possible), and thus, using \eqref{eq:ripjrip}, we obtain:
		$$w_{n+1} = b(ajb)^{l-1} aja(bja)^{l-1} bjb (ajb)^{l-1} aja(bja)^{l-1} b.$$
		By simplifying the form of $w_{n+1}$ and changing the letters to $i,j,k$ so that they appear in this given order, we obtain the form 2.a. of $w_{n+1}$.
		
			\item $i,j,m$ are pairwise different, then, by the same approach, we obtain:
			\begin{itemize}
				\item $p = (jmi)^l$ and $w_{n+1}=(imj)^l i(jmi)^l j(imj)^l i(jmi)^l$, which leads to the form 2.b. of $w_{n+1}$.			
			
				\item $p = m(jim)^{l-1}$, $w_{n+1}=(mij)^{l-1} mim (jim)^{l-1} j(mij)^{l-1} mim (jim)^{l-1}$, which gives the form 2.c. of $w_{n+1}$.			
			\end{itemize}
		\end{enumerate}

	\item $\vartheta_1 = E_m$, $\vartheta = R$, $\vartheta_2 = E_m$:
	
	Using \eqref{eq:tipjtittp}, we obtain:
	
	\begin{equation}
	\label{eq:eipjeip}
		w_{n+1} = E_m(p)ipjE_m(p)E_m(i)p = E_j(p)jpjE_j(p)jp,
	\end{equation}
	where $p$ is an $R$-palindrome. We used the fact that $j$ is a central factor of an $E_m$-palindrome and that $ipj$ is a central factor of an $R$-palindrome in order to derive the latter form.
	
	If $pj$ is a prefix of $a^{\omega}$, then $a=j$ and $pj$ is an $E_j$-palindromic suffix of $w_nj$ longer than $j$, which is a contradiction with $w_{n+1} = w_n j E_j(w_n)$.
	
	If $pj$ is a prefix of $(ab)^{\omega}$, then using the fact that $p=R(p)$, we have $p = a(ja)^{l-1}$. Applying \eqref{eq:eipjeip}, we obtain
	$$ w_{n+1} = b(jb)^{l-1} ja(ja)^{l-1} j b(jb)^{l-1}ja(ja)^{l-1}, $$
	which leads to the form 2.d. of $w_{n+1}$.
	
	If $pj$ is a prefix of $(abc)^{\omega}$, then $p$ cannot be an $R$-palindrome.
	
	\item $\vartheta_1 = E_m$, $\vartheta = E_r$, $\vartheta_2 = E_s$:
	
	Now, we obtain by \eqref{eq:tipjtittp}
	\begin{equation}
	\label{eipjekeep}
	w_{n+1} = E_m(p)ipjE_s(p) E_s(i) E_sE_m(p) = E_i(p) i p j E_j(p) r E_i(p),
	\end{equation}	
	where we used the fact that $j$ is a central factor of an $E_s$-palindrome and $ipj$ is a central factor of an $E_r$-palindrome.
	
	If $pj$ is a prefix of $a^{\omega}$, then it is a contradiction with the fact that $p$ is an $E_r$-palindrome with $r \neq j$.
	
	If $pj$ is a prefix of $(ab)^{\omega}$, then $p = (ji)^l$ because $p$ is an $E_r$-palindrome and
	$$w_{n+1} = (ir)^l i (ji)^l j (rj)^l r (ir)^l,$$ which corresponds to the form 2.e. of $w_{n+1}$.
	
	And finally, if $pj$ is a prefix of $(abc)^{\omega}$, then since $p$ is an $E_r$-palindrome, it can be of the form $p = (jri)^l$ or $p = r(jir)^l$. Neither of them is possible because $j$ is not the longest $E_j$-palindromic suffix of the resulting $w_nj$.
	\end{itemize}
	
\end{enumerate}
\end{proof}

\begin{proposition} \label{L_1_2}
Suppose that exactly one pseudopalindromic prefix was missed between $w_n$ and $w_{n+1}$ such that Assumption~\ref{assumption} holds and suppose that the prefix of length $n$ of the directive bi-sequence is normalized. Then:
\begin{enumerate}
\item
If $w_n = w_{n-1}ij\vartheta_1(w_{n-1})$ for two different $i, j \in \{0,1,2\}$ and $\vartheta_1 \in \{E_0, E_1, E_2, R\}$, then $w_{n+1} = w_nkl\vartheta_2(w_n)$ for two different $k,l \in \{0, 1, 2\}$ and $\vartheta_2 \in \{E_0, E_1, E_2, R\}$.
\item
If $w_{n+1} = w_nkl\vartheta_2(w_n)$ for two different $k,l \in \{0,1,2\}$ and $\vartheta_2 \in \{E_0, E_1, E_2, R\}$, then either $w_n = w_{n-1}ij\vartheta_1(w_{n-1})$ for two different $i,j \in \{0, 1, 2\}$ and $\vartheta_1 \in \{E_0, E_1, E_2, R\}$, or $w_{n+1}$ is of one of the following forms for $l\geq 1$ ($w_n=w_{n-1}\vartheta_1(w_{n-1})$ in the first four cases and $w_n=w_{n-1}(ab)^{-1}\vartheta_1(w_{n-1})$ for some $a,b \in \{0,1,2\}$ in the last two cases):

\begin{table}[ht!]
\catcode`\-=12
\centering
\bgroup
\def\arraystretch{1.2}

\begin{tabular}{l|l||l|l|}
\cline{2-4}
\multicolumn{1}{c|}{}& $w_{n+1}$              		& $w_n$       		& Missed pseudopal. prefix\\
\hhline{~===}
a.&$i^lj^{l+1}i^{l+1}j^l$                		& $i^lj^l$         & $i^lj^{l+1}i^l$                       \\ \cline{2-4}
b.&$i(ji)^ljk(ik)^likji(ji)^ljk(ik)^li$      	& $i(ji)^ljk(ik)^li$        & $i(ji)^ljk(ik)^likji(ji)^l$                \\ \cline{2-4}
c.&$i^lj^{l+1}k^{l+1}i^{l}$                 	& $i^lj^l$         & $i^lj^{l+1}k^l$                       \\ \cline{2-4}
d.&$(ij)^l ik(jk)^l ji(ki)^l kj(ij)^l$      	& $(ij)^l ik (jk)^l$    & $(ij)^l ik(jk)^l ji(ki)^l$                  \\ \cline{2-4}
e.&\makecell[lr]{$(ijkj)^{l-1} ijki(kjki)^{l-1} kj \ldots$ \\ $\ldots (ijkj)^{l-1}ijki(kjki)^{l-1}$}    & \makecell[lr]{$(ijkj)^{l-1} ijki\ldots$ \\ $\ldots (kjki)^{l-1}$}   & \makecell[lr]{$(ijkj)^{l-1} ijki(kjki)^{l-1}\ldots$ \\ $ \ldots kj(ijkj)^{l-1} i$  }             \\ \cline{2-4}
f.&\makecell[lr]{$ij(kjij)^{l-1} kjik(ijik)^{l-1} iji\ldots$\\ $\ldots kjk(ikjk)^{l-1}ikji(jkji)^{l-1} jk$ }       & \makecell[lr]{$ij(kjij)^{l-1} kjik\ldots$ \\ $\ldots(ijik)^{l-1} ij$ }  & \makecell[lr]{$ij(kjij)^{l-1} kjik (ijik)^{l-1}\ldots$ \\$\ldots  iji kjk(ikjk)^{l-1} i$} \\ \cline{2-4}
\end{tabular}

\egroup

\end{table}

\end{enumerate}
\end{proposition}
\begin{proof}
The proof is analogous to the proof of Proposition~\ref{L_1_1}.
\begin{enumerate}
	\item For $w_{n-1}=\varepsilon$ and $w_n=ij$, the assumption that the prefix of length $n$ of the directive bi-sequence is normalized is not met. Consider further on $|w_{n-1}|\geq 1$. First, suppose that $w_n = w_{n-1} ij \vartheta_1 (w_{n-1})$ and $w_{n+1} =  w_n p_2^{-1} \vartheta_2 (w_n)$ for some $p_2 \in \{0,1,2\}^*$, $|p_2| \geq 4$. Then $w_{n-1}^{(p)}$ and $w_{n-1}^{(c)}$ overlap as in the proof of Lemma \ref{L_1_p} and thus it is a contradiction with the fact that the prefix of length $n$ of the directive bi-sequence is normalized.

	Since an image of $w_n =  w_{n-1} ij \vartheta_1 (w_{n-1})$ is a central factor of $w_{n+1}$, $|p_2|$ has to be even. Further, we want to eliminate the cases where $|p_2| \in \{0,2\}$, see Figure \ref{OLspec21}. In the first case, the prefix of $w_{n+1}$ of length $|w_{n-1}| + 1$ is a pseudopalindrome, in the second case, the prefix of $w_{n+1}$ of length $|w_{n-1}| + 2$ is a pseudopalindrome, and thus we have a contradiction with the fact that the prefix of length $n$ of the directive bi-sequence is normalized. Thus, only the case $w_{n+1} = w_n ij \vartheta_2(w_n)$ remains possible.
		
	\begin{figure}[ht!]
			\begin{center}
				\includegraphics[scale=1.1]{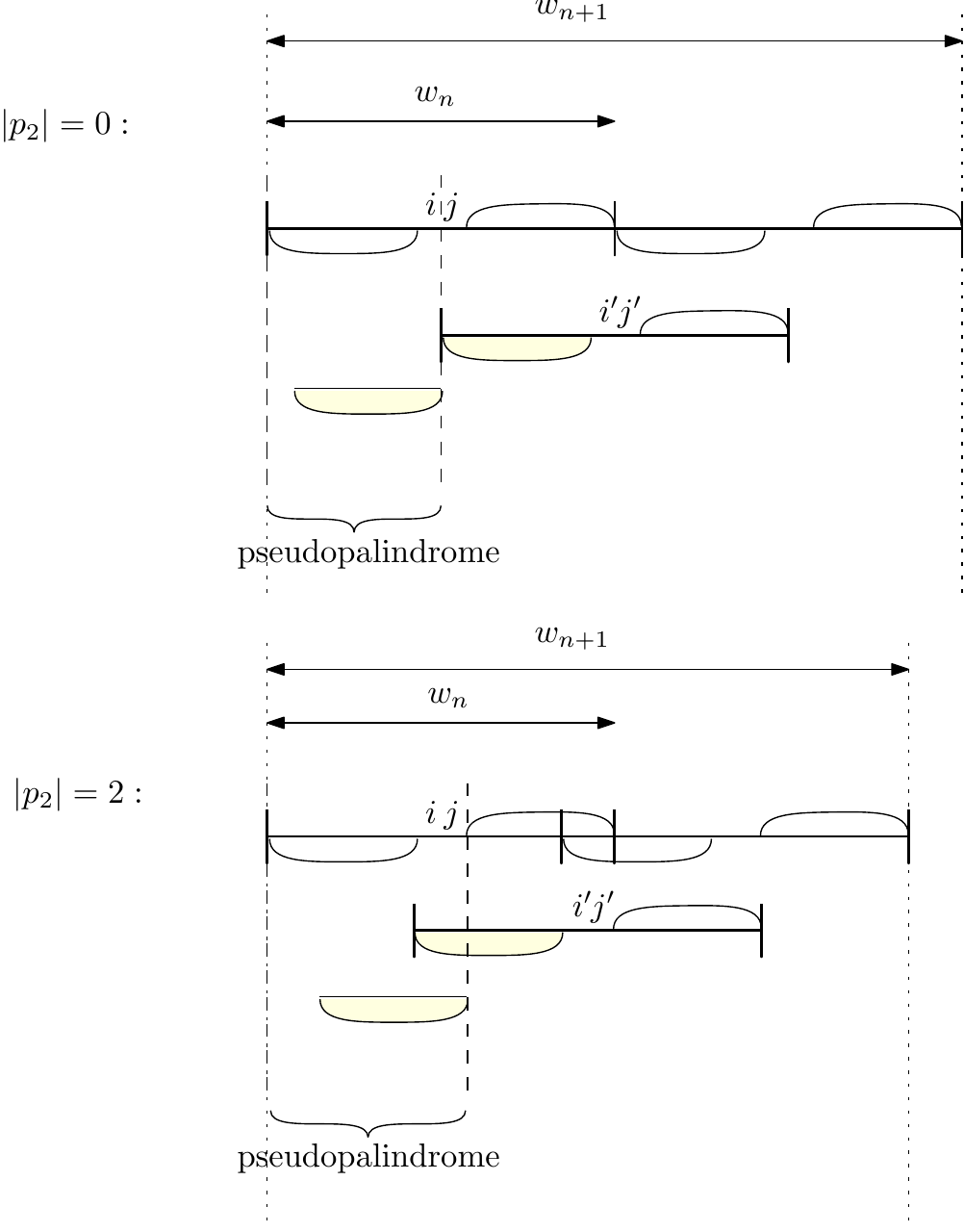}
			\end{center}
			\caption{Illustration of $w_{n+1}$ for $|p_2| \in \{0,2\}$.}
			\label{OLspec21}
		\end{figure}
	
	\item We will proceed exactly in the same way as in the second part of the proof of Proposition \ref{L_1_1}. We have $w_{n+1} = w_n st \vartheta_2 (w_n)$ for some distinct $s,t \in \{0,1,2\}$. We will consider all different possible forms of $w_n$. If $w_n  = w_{n-1} p_1^{-1} \vartheta_1(w_{n-1})$ for some $p_1 \in \{0,1,2\}^*$, $|p_1| \geq 4$, then the factor $w_{n-1}^{(p)}$ and $w_{n-1}^{(c)}$ again overlap, and thus two pseudopalindromic prefixes were missed between $w_n$ and $w_{n+1}$, which is not possible. Since the length of $p_1$ is even, the only remaining possibilities are $|p_1| \in \{0,2\}$ and $w_n = w_{n-1} ij \vartheta_1 (w_{n-1})$.
	The last case is possible and we will derive the special forms of $w_n$ from the other two cases:
	
	\begin{itemize}
	
		\item Let $w_n = w_{n-1}\vartheta_1 (w_{n-1})$ ($|p_1| = 0$). If $|w_{n-1}|=1$, $w_{n+1} \in \{ijjiij, ijjkki\}$. Now, we can suppose that $|w_{n-1}| \geq 2$. Let $w_n^{(0)}=w_n$, $w_n^{(0,1)}$, and $w_n^{(0,2)}=w_{n+1}$ be in order a $\vartheta_1$-, $\vartheta$-, and $\vartheta_2$-palindromes. Let $p$ denote the overlap of $w_n^{(0)}$ and $w_n^{(1)}$, see Figure \ref{OLspec22a}. The length of $p$ is equal to $|w_{n-1}| - 1$ and $p$ is again a $\vartheta$-palindrome by Observation \ref{o_overpal}. Moreover, $ps$ is a pseudopalindrome, too ($ps$ is an image of $w_{n-1}$), hence Lemma \ref{Ljednonavic} is applicable: $ps$ is a prefix of either $a^{\omega}$, $(ab)^{\omega}$, or $(abc)^{\omega}$ for some pairwise different $a,b,c \in \{0,1,2\}$.
		
		\begin{figure}[ht!]
			\begin{center}
				\includegraphics[scale=1.1]{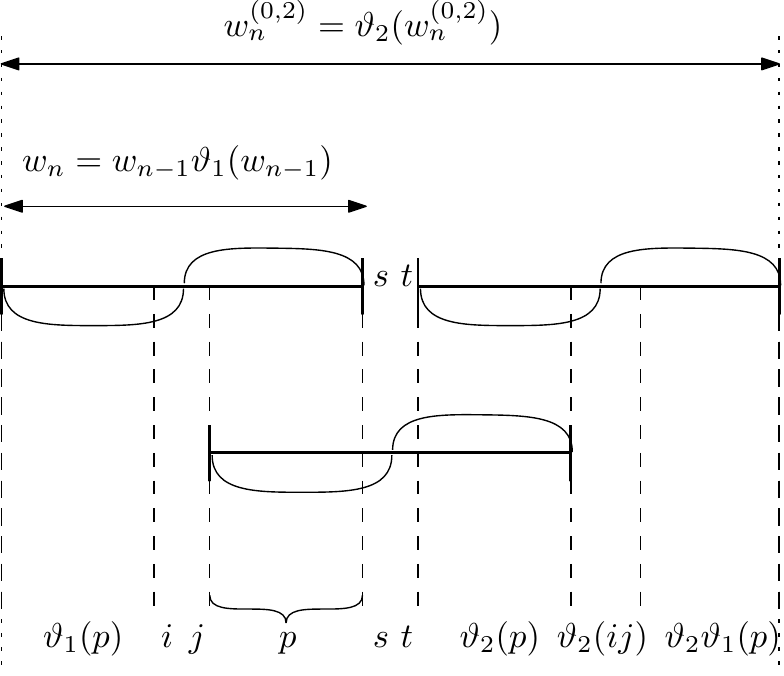}
			\end{center}
			\caption{Illustration for $w_n = w_{n-1}\vartheta_1 (w_{n-1})$.}
			\label{OLspec22a}
		\end{figure}
		
		Furthermore, $w_{n+1}$ satisfies the equation:
	
		\begin{equation}
		\label{eq:tipsttittp2}
			w_{n+1} = \vartheta_1(p) ij p s t \vartheta_2 (p) \vartheta_2 (ij)  \vartheta_2 \vartheta_1 (p).
		\end{equation}

		By Corollary \ref{CstairsType},  $\vartheta_1$, $\vartheta$, and $\vartheta_2$ are either successively $R$, $E_m$, $R$, or $E_m$, $R$, $E_m$, or $E_k$, $E_m$, $E_r$. The different cases will be addressed in the sequel:
		\begin{itemize}
			\item $\vartheta_1 = R$, $\vartheta = E_m$, $\vartheta_2 = R$:
			
			This case is not possible since $st$ cannot be a central factor of an $R$-palindrome for two different letters $s,t$.
			
			\item $\vartheta_1 = E_m$, $\vartheta = R$, $\vartheta_2 = E_m$:
			
			Using \eqref{eq:tipsttittp2}, we obtain:
			 $$ w_{n+1} = E_m(p)ijpstE_m(p)E_m(ij)p=E_m(p)ijpjiE_m(p)ijp, $$
			 where $i$, $j$, $m$ are pairwise different. We used the fact that $ijpst$ is a central factor of an $R$-palindrome.
			
			 Now, if $pj$ is a prefix of $a^{\omega}$, then $pj = j^{l+1}$ for some $l$, and using \eqref{eq:tipsttittp2}, we obtain $w_{n+1} = i^{l+1} j^{l+2} i^{l+2} j^{l+1}$, which is the prefix 2.a.
			
			 If $pj$ is a prefix of $(ab)^{\omega}$, then either $p = (ij)^l i$ and $pj$ is a non-empty $E_m$-palindromic suffix of $w_nj$, which is a contradiction with the fact that $w_{n+1} = w_n ji E_m(w_n)$. Or, $p = (mj)^{l}m$, which is possible, and we obtain the prefix 2.b.(when changing the letters to $i$, $j$, $k$ so that they appear in this given order).
			
			 The factor $pj$ cannot be a prefix of $(abc)^{\omega}$ because this word does not contain any $R$-palindrome (except of length $1$, which has been already examined above).
			
			\item $\vartheta_1 = E_k$, $\vartheta = E_m$, $\vartheta_2 = E_r$:
			
			Here, two cases can happen: either $m=i$ or $m = j$. Thus we obtain two possible equations from \eqref{eq:tipsttittp2}:

			\begin{align}
			\label{eq:dverovnice}
			w_{n+1} = E_k(p)ijpkiE_j(p)jkE_k(p) & \text{, where } p = E_i(p), \nonumber\\
			 & \text{or}  \\
			w_{n+1} = E_k(p)ijpjk E_i(p) ki E_k(p) & \text{, where } p = E_j(p).\nonumber
			\end{align}

			If $pk$ or $pj$ is a prefix of $a^{\omega}$, then only the case $p = j^l$ is possible and we obtain the form 2.c. ($p  = k^l$ is not possible since $k^l$ is not an $E_i$-palindrome).
			
			If $pk$ or $pj$ is a prefix of $(ab)^{\omega}$, then only $p = (kj)^l$ is possible, thus $w_{n+1} = (ik)^l ij (kj)^l ki (ji)^l jk (ik)^l$, which is the form 2.d.
			
			The case where $pj$ or $pk$ is a prefix of $(abc)^{\omega}$ cannot happen because such factors have a suffix of length three composed of three different letters. Therefore, $w_nk$, resp. $w_nj$ does not have an empty $E_j$, resp. $E_i$-palindromic suffix, which is a contradiction with the form of $w_{n+1}=w_nkiE_j(w_n)$, resp. $w_{n+1}=w_njkE_i(w_n)$.
		\end{itemize}

		\item Let now $w_n = w_{n-1} (ij)^{-1} \vartheta_1 (w_{n-1})$ ($|p_1|=2$). For $|w_{n-1}| = 2$, we have $|w_n| = 2$, which is not possible. We further consider $|w_{n-1}| > 2$. Let $w_n^{(0)}$, $w_n^{(0,1)}$, and $w_n^{(0,2)}$ be in order a $\vartheta_1$, $\vartheta$, and $\vartheta_2$-palindromes. Let $p$ denote again the overlap of $w_n^{(0)}$ and $w_n^{(1)}$, see Figure \ref{OLspec22b}. The length of $p$ is equal to $|w_{n-1}| -2$ and $p$ is a $\vartheta$-palindrome by Observation \ref{o_overpal}. Moreover, $pst$ is a pseudopalindrome, too ($pst$ is an image of $w_{n-1}$). Hence, Lemma \ref{Ldvenavic} is applicable: $pst$ is a prefix of $(ab)^{\omega}$, $(abba)^{\omega}$, $a(baca)^{\omega}$, $(abbcca)^{\omega}$, $a(bcba)^{\omega}$, $(aabb)^{\omega}$, $(aabbcc)^{\omega}$ for some pairwise distinct $a$, $b$, $c \in \{0,1,2\}$.
		
				\begin{figure}[ht!]
			\begin{center}
				\includegraphics[scale=1.1]{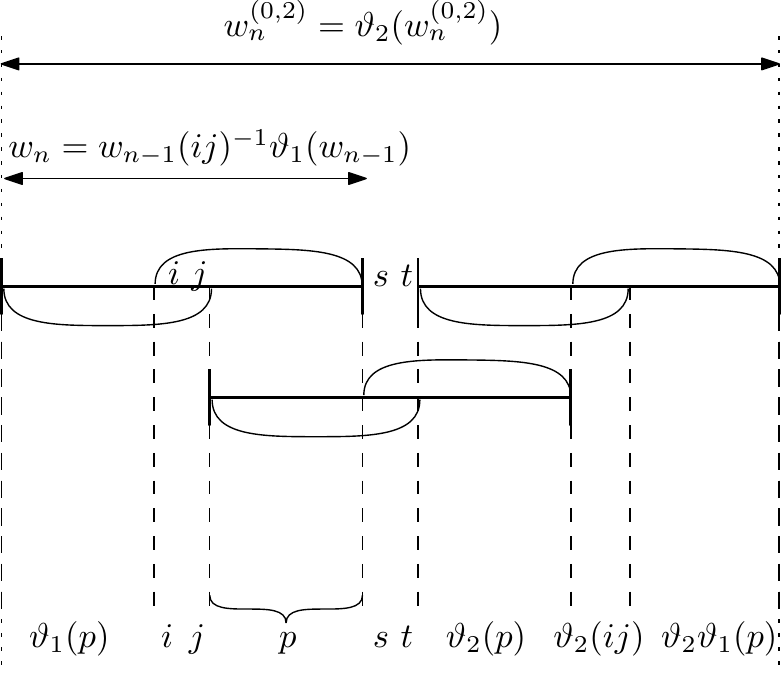}
			\end{center}
			\caption{Illustration for $w_n = w_{n-1} (ij)^{-1} \vartheta_1 (w_{n-1})$.}
			\label{OLspec22b}
		\end{figure}

Furthermore, $w_{n+1}$ satisfies \eqref{eq:tipsttittp2}.

By Corollary \ref{CstairsType}, $\vartheta_1$, $\vartheta$, and $\vartheta_2$ are either successively $R$, $E_m$, $R$, or $E_m$, $R$, $E_m$, or $E_m$, $E_r$, $E_q$. The different cases will be addressed in the sequel:
\begin{itemize}
\item $\vartheta_1 = R$, $\vartheta = E_m$, $\vartheta_2 = R$:

This case is not possible since $st$ cannot be a central factor of an $R$-palindrome for two different letters $s,t$.

\item $\vartheta_1 = E_m$, $\vartheta = R$, $\vartheta_2 = E_m$:

Using \eqref{eq:tipsttittp2}, we obtain:
			 $$ w_{n+1} = E_m(p)ijpjiE_m(p)ijp, $$
where, $i$, $j$, $m$ are pairwise different.

Since $p$ is an $R$-palindrome, $pji$ cannot be a prefix of $(abbcca)^{\omega}$, $(aabb)^{\omega}$, and $(aabbcc)^{\omega}$.

If $pji$ is a prefix of $(ab)^{\omega}$, then $p = (ij)^{l-1}i$.
If $pji$ is a prefix of $(abba)^{\omega}$, then $p = (jiij)^l$.
If $pji$ is a prefix of $a(baca)^{\omega}$, then $p = i(miji)^{l-1}mi$.
In all three previous cases, the $E_m$-palindromic suffix of $w_nj$ is non-empty, which is a contradiction with $w_{n+1} =w_n ji E_m(w_n)$.

If $pji$ is a prefix of $(abc)^{\omega}$, then $p = m$.
If $pji$ is a prefix of $a(bcba)^{\omega}$, then $p= m(jijm)^{l-1}$.
These two cases lead to the form 2.e.

\item $\vartheta_1 = E_m$, $\vartheta = E_r$, $\vartheta_2 = E_q$:

Using the fact that $st$ is a central factor of an $E_q$-palindrome and $ijpst$ is a central factor of an $E_r$-palindrome, we obtain two possible equations for $w_{n+1}$ from \eqref{eq:tipsttittp2}:

\begin{align}
			\label{eq:dverovnice2}
			w_{n+1} = E_m(p)ijpmiE_j(p)jmE_m(p) & \text{, where } p = E_i(p), \nonumber\\
			 & \text{or}  \\
			w_{n+1} = E_m(p)ijpjm E_i(p) mi E_m(p) & \text{, where } p = E_j(p).\nonumber
			\end{align}

Since $p$ is an $E_i$-, resp. $E_j$-palindrome, $pmi$, resp. $pjm$ cannot be a prefix of $(ab)^{\omega}$, $(abba)^{\omega}$, $(aabb)^{\omega}$, $(aabbcc)^{\omega}$.

If $pmi$, resp. $pjm$ is a prefix of $(abc)^{\omega}$, then $p = (mij)^l$, resp. $p = (mij)^{l-1}mi$, and $w_n m$, resp. $w_n j$ has a non-empty $E_j$, resp. $E_i$-palindromic suffix, which is a contradiction with $w_{n+1} =w_n miE_j(w_n)$, resp. $w_{n+1} = w_n jm E_i(w_n)$.

If $pmi$ is a prefix of $a(baca)^{\omega}$, then $p = i(miji)^{l-1}$ and $w_nm$ has a non-empty $E_j$-palindromic suffix.
Moreover, $pjm$ cannot be a prefix of $a(baca)^{\omega}$ for $p = E_j(p)$.	

If $pjm$ is a prefix of $(abbcca)^{\omega}$, then $p = (jmmiij)^l$ and $w_nj$ has a non-empty $E_i$-palindromic suffix. Moreover, $pmi$ cannot be a prefix of $(abbcca)^{\omega}$ for $p = E_i(p)$.

If $pjm$ is a prefix of $a(bcba)^{\omega}$, then $p = m(jijm)^{l-1}ji$ and we obtain the form 2.f. Moreover, $pmi$ cannot be a prefix of $a(bcba)^{\omega}$ for $p = E_i(p)$.

		\end{itemize}
		
	\end{itemize}
	
\end{enumerate}

\end{proof}

\subsubsection{Normalization rules}
At this point, it is necessary to mention that the normalized form of a directive bi-sequence is not always unique over a ternary alphabet. The directive bi-sequence $\DT$, where $\Delta=i^l$ and $\Theta \in \{R, E_i\}^*$, is normalized and generates the word $i^l$. It is easily seen that if a prefix $w_n$ of $\mbu \DT$ contains two different letters, then the normalized sequence is defined uniquely starting from the index $n$.
\begin{remark}\label{preprocessing}
Suppose that $i^l$ is the longest prefix of $\mbu \DT$ that contains only the letter $i$. From now on, we will say that the bi-sequence $\DT$ is normalized if $\DT$ is normalized according to Definition \ref{Dnormalized} and the prefix of $\Theta$ is $E_i^l$. This preprocessing of the prefix of the directive bi-sequence will be done before starting the normalization process.
\end{remark}
\begin{example}
The directive bi-sequence $(0^{\omega}, E_0^{\omega})$ is normalized. The directive bi-sequence $(0001^{\omega}, RE_0RE_2^{\omega})$ is not, its normalized form is $(0001^{\omega}, E_0E_0E_0E_2^{\omega})$ and we directly see that both of them generate the same generalized pseudostandard word.
\end{example}

\noindent {\bf Prefix rules}

\noindent Now that every generalized pseudostandard word has exactly one normalized directive bi-sequence, we will derive prefix substitution rules for cases where one pseudopalindromic prefix was missed from Propositions \ref{Ospecwijw}, \ref{Ospecwiw}, \ref{L_1_1}, and \ref{L_1_2}. These rules define how to rewrite prefixes of $\DT$ so as not to miss any pseudopalindromic prefix. On the left, the prefix of length $n$ of $(\Delta, \Theta)$ is normalized and there is one missed pseudopalindromic prefix between $w_n$ and $w_{n+1}$. On the right, the prefix of $(\Delta, \Theta)$ is rewritten so that the same prefix of a pseudostandard word is obtained and that the prefix of $(\Delta, \Theta)$ of length $n+1$ is normalized. The index $l$ in the rules can take any positive integer value.

First, the special forms of $w_{n+1}$ from Propositions \ref{Ospecwijw}, \ref{Ospecwiw}, \ref{L_1_1}, and \ref{L_1_2} are considered one by one. Their corresponding non-normalized and normalized directive bi-sequence is found, followed by the new prefix substitution rule obtained:

\begin{itemize}

\item $w_{n+1} = i^{l}j^{l}$

The normalized bi-sequence of $w_n = i^{l-1}$ is $(i^l, E_i^l)$. If we want to obtain directly $w_{n+1}$, then $\delta_{n+1}$ is $i$ and $\vartheta_{n+1}$ is $E_k$. Furthermore, we know that we missed the pseudopalindromic prefix $i^l$. Thus, the non-normalized bi-sequence of $w_{n+1}$ is $(i^l, E_i^{l-1}E_k)$ and the normalized bi-sequence is $(i^lj, E_i^lE_k)$. We obtain the new prefix rule:

\begin{prefixrule}\
\label{p1}
 (i^{l}, E_i^{l-1}E_k) \rightarrow (i^{l}j, E_i^{l}E_k).
\end{prefixrule}

\item $w_{n+1} = (ij)^l(ki)^l$ :

\begin{prefixrule}\
\label{p3}
 (i(ji)^{l-1}j, E_i(E_kR)^{l-1}E_i) \rightarrow (i(ji)^{l-1}jk, E_i(E_kR)^{l-1}E_kE_i),
\end{prefixrule}

for $l \geq 2$ (for $l=1$ two pseudopalindromic prefixes were missed).

\item $w_{n+1} = (ij)^lik(jk)^l$:
\begin{prefixrule}\
\label{p4}
 ((ij)^{l}i, E_iE_k(RE_k)^{l-1}E_j) \rightarrow ((ij)^{l}ik, E_iE_k(RE_k)^{l-1}RE_j).
\end{prefixrule}

\item $w_{n+1} = ijjkki$:
\begin{prefixrule}\
\label{p6}
 (ijj, E_iE_kE_i) \rightarrow (ijjk,E_iE_kE_jE_i).
\end{prefixrule}
\item $w_{n+1} = ijjiij$:
\begin{prefixrule}\
\label{p7}
 (ijj, E_iE_kE_k) \rightarrow (ijji,E_iE_kRE_k).
\end{prefixrule}
\item $w_{n+1} = (ijkj)^l(ikij)^l$:
\begin{prefixrule}\
\label{p8}
(ijk(jj)^{l-1}j, E_iE_kE_j(RE_j)^{l-1}E_k) \rightarrow (ijk(jj)^{l-1}jk, E_iE_kE_j(RE_j)^{l-1}RE_k).
\end{prefixrule}
\item $w_{n+1} = (ijkj)^lijki(kjki)^l$:
\begin{prefixrule}\
\label{p9}
\begin{split}
 (ijkj(jj)^{l-1}j, E_iE_kE_jR(E_jR)^{l-1}E_i) \rightarrow \\ (ijkj(jj)^{l-1}ji, E_iE_kE_jR(E_jR)^{l-1}E_jE_i).
 \end{split}
\end{prefixrule}

\item $w_{n+1} = iji$:
\begin{prefixrule}\
\label{p10}
 (ij, E_iR) \rightarrow (iji, E_iE_kR).
\end{prefixrule}

\item $w_{n+1} = ijk$:
\begin{prefixrule}\
\label{p11}
 (ij, E_iE_j) \rightarrow (ijk, E_iE_kE_j).
\end{prefixrule}
\item $w_{n+1} = (ij)^li(ki)^l$:
\begin{prefixrule}\
\label{p12}
 (ij(ij)^{l-1}i, E_iE_k(RE_k)^{l-1}E_i) \rightarrow (ij(ij)^{l-1}ik, E_iE_k(RE_k)^{l-1}RE_i).
\end{prefixrule}
\item $w_{n+1} = (ij)^l ijk(jk)^l$:
\begin{prefixrule}\
\label{p13}
 (i(ji)^{l}j,E_i(E_kR)^{l}E_j) \rightarrow (i(ji)^{l}jk, E_i(E_kR)^{l}E_kE_j).
\end{prefixrule}

\item $w_{n+1} = (ijk)^{l-1} ijkji(kji)^{l-1}$:
\begin{prefixrule}\
\label{p15}
 ((ijk)^{l-1}ijk, (E_iE_kE_j)^{l-1}E_iE_kR) \rightarrow ((ijk)^{l}j, (E_iE_kE_j)^{l}R).
\end{prefixrule}
\item $w_{n+1} = (ijk)^l iji(kji)^l$:
\begin{prefixrule}\
\label{p16}
 ((ijk)^{l}ij, (E_iE_kE_j)^{l}E_iR) \rightarrow ((ijk)^{l}iji, (E_iE_kE_j)^{l}E_iE_kR).
\end{prefixrule}
\item $w_{n+1} = (ijk)^l i(kji)^l$:
\begin{prefixrule}\
\label{p17}
 ((ijk)^{l}i, (E_iE_kE_j)^{l}R) \rightarrow ((ijk)^{l}ik, (E_iE_kE_j)^{l}E_iR).
\end{prefixrule}

\item $w_{n+1} = (ijk)^lk^{-1}(kji)^lk(ijk)^lk^{-1}(kji)^l$:
\begin{prefixrule}\
\label{p18}
 ((ijk)^{l}jk, (E_iE_kE_j)^{l}RR) \rightarrow ((ijk)^{l}jkk, (E_iE_kE_j)^{l}RE_kR).
\end{prefixrule}
\item $w_{n+1} = (ijk)^li(kji)^lk(ijk)^li(kji)^l$:
\begin{prefixrule}\
\label{p19}
 ((ijk)^{l}ikk, (E_iE_kE_j)^{l}E_iRR) \rightarrow ((ijk)^{l}ikki, (E_iE_kE_j)^{l}E_iRE_jR).
\end{prefixrule}
\item $w_{n+1} = (ijk)^{l-1} iji(kji)^{l-1} k(ijk)^{l-1} iji(kji)^{l-1}$:
\begin{prefixrule}\
\label{p20}
\begin{split}
 ((ijk)^{l-1}ijik, (E_iE_kE_j)^{l-1}E_iE_kRR) \rightarrow \\ ((ijk)^{l-1}ijikj, (E_iE_kE_j)^{l-1}E_iE_kRE_iR).
 \end{split}
\end{prefixrule}
\item $w_{n+1} = i(ji)^{l-1} jk(jk)^{l-1} j(ij)^{l-1} ij(kj)^{l-1} k $:
\begin{prefixrule}\
\label{p21}
 (i(ji)^{l-1}jkj, E_i(E_kR)^{l-1}E_kE_jE_j) \rightarrow (i(ji)^{l-1}jkjj, E_i(E_kR)^{l-1}E_kE_jRE_j).
\end{prefixrule}
\item $w_{n+1} = (ij)^l i(ki)^l k(jk)^l j(ij)^l$:
\begin{prefixrule}\
\label{p22}
\begin{split}
(ij(ij)^{l-1}ikk, E_iE_k(RE_k)^{l-1}RE_iE_k) \rightarrow \\ (ij(ij)^{l-1}ikkj, E_iE_k(RE_k)^{l-1}RE_iE_jE_k).
\end{split}
\end{prefixrule}

\item $w_{n+1} = i^lj^{l+1}i^{l+1}j^l$:
\begin{prefixrule}\
\label{p23}
 (i^{l}jj, E_i^{l}E_kE_k) \rightarrow (i^{l}jji, E_i^{l}E_kRE_k).
\end{prefixrule}

\item $w_{n+1} = i(ji)^ljk(ik)^likji(ji)^ljk(ik)^li$:
\begin{prefixrule}\
\label{p24}
 (ij(ij)^{l}kk, E_iE_k(RE_k)^{l}E_iE_i) \rightarrow (ij(ij)^{l}kkj, E_iE_k(RE_k)^{l}E_iRE_i).
\end{prefixrule}

\item $w_{n+1} = i^lj^{l+1}k^{l+1}i^{l}$:
\begin{prefixrule}\
\label{p25}
 (i^{l}jj, E_i^{l}E_kE_i) \rightarrow (i^{l}jjk, E_i^{l}E_kE_jE_i).
\end{prefixrule}

\item $w_{n+1} = (ij)^l ik(jk)^l ji(ki)^l kj(ij)^l$:
\begin{prefixrule}\
\label{p26}
\begin{split}
 (ij(ij)^{l-1}ikj, E_iE_k(RE_k)^{l-1}RE_jE_k) \rightarrow \\ (ij(ij)^{l-1}ikjk, E_iE_k(RE_k)^{l-1}RE_jE_iE_k).
 \end{split}
\end{prefixrule}

\item $w_{n+1} = (ijkj)^{l-1} ijki(kjki)^{l-1} kj(ijkj)^{l-1} ijki(kjki)^{l-1}$:
\begin{prefixrule}\
\label{p27}
\begin{split}
 (ijk(jj)^{l-1}ik, E_iE_kE_j(RE_j)^{l-1}E_iE_i) \rightarrow \\ (ijk(jj)^{l-1}ikj, E_iE_kE_j(RE_j)^{l-1}E_iRE_i).
\end{split}
\end{prefixrule}

\item $w_{n+1} = ij(kjij)^{l-1} kjik(ijik)^{l-1} ijikjk(ikjk)^{l-1} ikji(jkji)^{l-1} jk$:
\begin{prefixrule}\
\label{p28}
\begin{split}
 (ijk(jj)^{l-1}jki, E_iE_kE_j(RE_j)^{l-1}RE_kE_j) \rightarrow \\ (ijk(jj)^{l-1}jkik, E_iE_kE_j(RE_j)^{l-1}RE_kE_iE_j).
\end{split}
\end{prefixrule}

\end{itemize}

\noindent {\bf Factor rules}

\noindent The next theorem concludes the section concerning one pseudopalindromic prefix being missed between $w_n$ and $w_{n+1}$. Three factor substitution rules are obtained. Those three rules contain factors of the directive bi-sequence that are not normalized on the left, and their normalized transcription on the right.
\begin{theorem}
\label{nonprefix1}
Let $(\Delta,\Theta) = (\delta_1\delta_2\ldots, \vartheta_1\vartheta_2\ldots)$ be a directive bi-sequence having a normalized prefix of length $n$. Moreover, let the prefix of $(\Delta,\Theta)$ of length $n+1$ be different from any prefix on the left side of the prefix rules \eqref{p1} to \eqref{p28}. Then there is exactly one missed pseudopalindromic prefix between $w_n$ and $w_{n+1}$ if, and only if, $(\delta_{n-1}\delta_n\delta_{n+1}, \vartheta_{n-1}\vartheta_n\vartheta_{n+1})$ has one of the following forms:
\begin{flalign} \label{eq:rule123}
\begin{split}
\bullet& \; \; (ab_1b_2, RE_iE_i),\text{ where } b_1 = E_i(b_2), \text{ }(\text{except } (iii, RE_iE_i) \text{ for } w_{n-1}=i^{n-1}),\\
\bullet& \; \; (ab_1b_2, E_iRR),\text{ where } b_1 = E_i(b_2), \text{ }(\text{except } (iii, E_iRR) \text{ for } w_{n-1}=i^{n-1}),\\
\bullet& \; \; (ab_1b_2, E_iE_jE_i),\text{ where } E_i(b_1)= E_j(b_2).
\end{split}
\end{flalign}

Therefore, we obtain a set of factor substitution rules (not necessarily applicable to a prefix):
\begin{enumerate}
\item
$(ab_1b_2, RE_iE_i) \rightarrow (ab_1b_2b_1, RE_iRE_i),$ where $b_1 = E_i(b_2)$,
\item
$(ab_1b_2, E_iRR) \rightarrow (ab_1b_2b_1, E_iRE_iR),$ where $b_1 = E_i(b_2)$,
\item
$(ab_1b_2, E_iE_jE_i) \rightarrow (ab_1b_2E_iE_j(b_1),E_iE_jE_kE_i),$ where $E_i(b_1)= E_j(b_2)$.
\end{enumerate}

\end{theorem}

\begin{proof}
$(\Rightarrow):$ From Lemma~\ref{L_1_p} and Propositions \ref{L_1_1} and \ref{L_1_2}, one of the following possibilities holds:
\begin{itemize}
\item
$w_{n+1} = w_np_2^{-1}\theta_2(w_n)$ and $w_n = w_{n-1}p_1^{-1}\theta_1(w_{n-1})$, where $|p_1| = |p_2|$,
\item
$w_{n+1} = w_ni\theta_2(w_n)$ and $w_n = w_{n-1}j\theta_1(w_{n-1})$,
\item
$w_{n+1} = w_nij\theta_2(w_n)$ and $w_n = w_{n-1}kl\theta_1(w_{n-1})$.
\end{itemize}
Hence, $w_{n+1}$ is of one of the following forms:
\begin{flalign} \label{eq:form1missed}
\begin{split}
\bullet& \; \; w_{n+1} = w_{n-1}p_1^{-1}\theta_1(w_{n-1})p_2^{-1}\theta_2\theta_1(w_{n-1})\theta_2(p_1)^{-1}\theta_2(w_{n-1}), |p_1| = |p_2|, \\[2.5mm]
\bullet& \; \; w_{n+1} = w_{n-1}j\theta_1(w_{n-1})i\theta_2\theta_1(w_{n-1})\theta_2(j)\theta_2(w_{n-1}), \\[2.5mm]
\bullet& \; \; w_{n+1} = w_{n-1}kl\theta_1(w_{n-1})ij\theta_2\theta_1(w_{n-1})\theta_2(kl)\theta_2(w_{n-1}).
\end{split}
\end{flalign}

The rest of the proof will be focused on the first form of $w_{n+1}$, the second and third case can be treated analogously. The assumptions of Theorem \ref{nonprefix1} guarantee that Assumption~\ref{assumption} is met, thus $w_n^{(1))}$ is a central factor of $w_{n+1}$. We can deduce from the given form of $w_{n+1}$ in \eqref{eq:form1missed} that the missed pseudopalindromic prefix $w_n^{(0,1)}$ is
$w_{n-1}p_1^{-1}\theta_1(w_{n-1})p_2^{-1}\theta_2\theta_1(w_{n-1})$. Moreover, since $\theta_1 (w_{n-1})$ is its central factor, it is of the same pseudopalindromic type as $w_n^{(0,1)}$. See Figure \ref{OT1missed} for a better understanding.

	\begin{figure}[ht!]
		\begin{center}
			\includegraphics[scale=1.1]{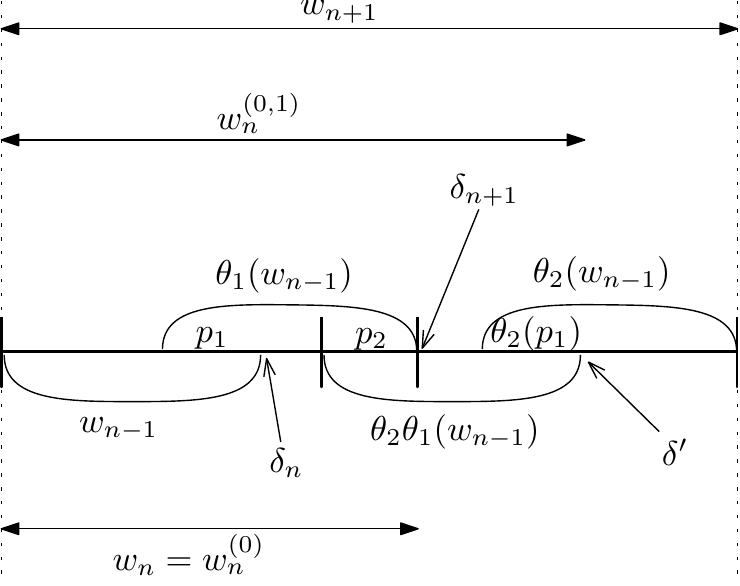}
		\end{center}
		\caption{Illustration of the first form of $w_{n+1}$.}
		\label{OT1missed}
	\end{figure}
	
Since one pseudopalindromic prefix was missed between $w_n$ and $w_{n+1}$, by Corollary \ref{CstairsType}, $\theta_1 = \theta_2 = R$, $\theta_1 = \theta_2 = E_i$, or $\theta_1 = E_j$ and $\theta_2 = E_i$. In the first case, an $E_i$-palindrome was missed between $w_n$ and $w_{n+1}$, in the second case an $R$-palindrome, and in the third case an $E_k$-palindrome.

\begin{itemize}
	\item $\theta_1 = \theta_2 = R$:
	
Since $w_n = (w_{n-1}\delta_n)^R$ is an $R$-palindrome, $R(\delta_n) = \delta_n$ is the letter preceding $w_n^{(1)}$. Since, $w_n^{(0,1)}$ is a (missed) $E_i$-palindromic prefix, $\delta_{n+1} = E_i (\delta_n)$. Moreover, $\theta_1(w_{n-1}) = R(w_{n-1})$ is an $E_i$-palindrome, and thus $w_{n-1}$ is an $E_i$-palindrome by Observation \ref{o_natureofpalonpal}.
	
	Overall, $w_{n-1} = E_i(w_{n-1})$, $w_n$ = $R(w_n)$, $w_{n+1} = R(w_{n+1})$, where $\delta_n = E_i (\delta_{n+1})$. The letter $\delta '$ following $w_n^{(0,1)}$ is $R(R(\delta_n)) = \delta_n$.
Using the notation $\delta_{n-1}=a$, $\delta_n=b_1$, $\delta_{n+1}=b_2$ and $\vartheta_{n-1}=E_i$, $\vartheta_{n}=R$, $\vartheta_{n+1}=R$, we obtain the rule 2.	
	
	\item $\theta_1 = \theta_2 = E_i$:
	Similarly, we obtain that $w_{n-1} = R(w_{n-1})$, $w_n$ = $E_i(w_n)$, $w_{n+1} = E_i(w_{n+1})$, where $\delta_n = E_i (\delta_{n+1})$. The missed pseudopalindromic prefix is an $R$-palindrome and the letter $\delta '$ following it is $E_i(E_i(\delta_n)) = \delta_n$. Consequently, we have the rule 1.
	
	\item $\theta_1 = E_j$ and  $\theta_2 = E_i$:
	Here, $w_{n-1} = E_i(w_{n-1})$, $w_n$ = $E_j(w_n)$, $w_{n+1} = E_i(w_{n+1})$. Since $w_n^{(0,1)}$ is a (missed) $E_k$-palindromic prefix, $E_j(\delta_n)=E_k(\delta_{n+1})$, which is equivalent to $E_i(\delta_n)=E_j(\delta_{n+1})$. The letter $\delta '$ following $w_n^{(0,1)}$ is $E_i(E_j(\delta_n))$. It corresponds to the rule 3.
\end{itemize}

$(\Leftarrow):$ It is easily seen that if $(\delta_{n-1}\delta_n\delta_{n+1}, \vartheta_{n-1}\vartheta_n\vartheta_{n+1})$ has one of the form in \eqref{eq:rule123}, then $w_{n+1}$ has one of the forms in \eqref{eq:form1missed}. In order to have this implication, it was necessary to exclude $(iii, E_iRR)$ and $(iii, RE_iE_i)$ for $w_{n-1}=i^{n-1}$. In the first case, the (only) missed pseudopalindromic prefix between $w_n$ and $w_{n+1}$ is the factor $w_n^{(0,1)} = w_{n-1}p_1^{-1}\theta_1(w_{n-1})p_2^{-1}\theta_2\theta_1(w_{n-1})$. Similarly, for the remaining two cases.

Finally, we obtain the non-prefix rules by knowing $\delta'$ and what type of pseudopalindrome was missed.

\begin{enumerate}
\item
$(ab_1b_2, RE_iE_i) \rightarrow (ab_1b_2b_1, RE_iRE_i),  \text{ where } b_1 = E_i(b_2)$,
\item
$(ab_1b_2, E_iRR) \rightarrow (ab_1b_2b_1, E_iRE_iR), \text{ where } b_1 = E_i(b_2)$,
\item
$(ab_1b_2, E_iE_jE_i) \rightarrow (ab_1b_2E_iE_j(b_1), \text{ where } E_iE_jE_kE_i), E_i(b_1)= E_j(b_2)$.
\end{enumerate}

\end{proof}

The rules obtained in Theorem \ref{nonprefix1} are applicable to any factor of the directive bi-sequence including the prefix. Since we decided that a normalized directive bi-sequence $\DT$ has $(i^l, E_i^l)$ as prefix whenever $i^l$ is a prefix of $\mbu \DT$, it is possible that the $R$-palindromic closure of the rule $1$, resp. $2$ in Theorem \ref{nonprefix1} has been replaced by the antimorphism $E_i$ during the preprocessing procedure. For example, if the beginning of $\DT$ is $(00000020, RE_0RRE_0RE_1E_1)$, then we process it to $(00000020, E_0E_0E_0E_0E_0E_0E_1E_1)$ and the factor rule is not applicable anymore even if it should because the palindromic prefix $000000222222000000$ was missed. This situation will be prevented by adding three more additional rules that we will now derive taking into account the possible forms of the rule 1, resp. 2:

\begin{enumerate}
\item $(iii, RE_iE_i)$ or $(iii, E_iRR)$:
If the prefix of $\DT$ is $(i^{l}ii, \{R,E_i\}^{l} E_iE_i$), $l \geq 1$, or $(i^{l}ii, \{R,E_i\}^lRR)$, $l\geq 1$, then it is normalized and no additional prefix rule is needed.

\item $(iij, RE_kE_k)$:
This leads to the prefix of $\DT$ equal to $(i^{l-1} i ij, E_i^{l-1} E_i E_k E_k)$. In this case, the prefix is not normalized already between the last antimorphism $E_i$ and the first antimorphism $E_k$ (the prefix rule \eqref{p1} is applicable), so no additional prefix rule is needed.

\item $(iji, RE_kE_k)$:
Here, if the prefix $\DT$ is $(i^{l}ji, \{R,E_i\}^{l} E_kE_k)$, $l \geq 1$, then we transform $\DT$ to $(i^{l}ji, E_i^{l} E_kE_k)$ and we obtain a new prefix rule:
\begin{prefixrule}\
\label{p29}
 (i^{l}ji, E_i^{l} E_kE_k) \rightarrow (i^{l}jij, E_i^{l} E_kRE_k).
\end{prefixrule}

\item $(ijj, RE_jE_j)$: If the prefix of $\DT$ is $(ijj, E_iE_jE_j)$ or $(ijj, R E_j E_j)$, then $(ij, E_iE_j)$ is already not normalized and the prefix rule \eqref{p11} is applicable, hence no new prefix rule is needed.
On the other hand, if the prefix of $\DT$ is $(i^{l+1}jj, \{R,E_i\}^{l+1} E_jE_j)$, $l \geq 1$, then the factor rule is applicable and we obtain a new prefix rule:
\begin{prefixrule}\
\label{p30}
 (i^{l}ijj, E_i^{l}E_iE_jE_j) \rightarrow (i^{l}ijjj, E_i^{l}E_iE_jRE_j).
\end{prefixrule}

\item $(ijk, RE_iE_i)$:
With similar arguments as in the case of $(ijj, RE_jE_j)$, we obtain a new prefix rule:
\begin{prefixrule}\
\label{p31}
 (i^{l}ijk, E_i^{l}E_iE_iE_i) \rightarrow (i^{l}ijkj, E_i^{l}E_iE_iRE_i).
\end{prefixrule}

\end{enumerate}

Now that we solved the case where one pseudopalindromic prefix was missed between $w_n$ and $w_{n+1}$, we will examine the remaining case where two pseudopalindromic prefixes were missed.

\subsection{Missing two pseudopalindromic prefixes}
In this section, we suppose that Assumption~\ref{assumption} is satisfied and that we missed exactly two pseudopalindromic prefixes between $w_n$ and $w_{n+1}$.

We are interested only in the cases where $w_n^{(0)}$ overlaps with $w_n^{(2)}$. If $w_n^{(0)}$ overlaps with $w_n^{(3)}$, then clearly all images of $w_n$ overlap pairwise. If $w_{n+1} =w_n\vartheta_1(w_n)$, then $w_n^{(0)}$ and $w_n^{(2)}$ also overlap. If $w_{n+1} = w_nij\vartheta_1(w_n)$, resp. $w_{n+1} = w_ni\vartheta_1(w_n)$, then the cases where $w_n^{(0)}$ and $w_n^{(2)}$ do not overlap are treated in Proposition \ref{Ospecwijw}, resp. Proposition \ref{Ospecwiw}.

\begin{definition}
Suppose that exactly two palindromic prefixes were missed between the prefixes $w_n$ and $w_{n+1}$ such that $w_n^{(0)}$ and $w_n^{(2)}$ overlap. Furthermore, suppose that the prefix of length $n$ of the directive bi-sequence $\DT$ is normalized. Then the overlap of $w_n^{(i)}$ and $w_n^{(i+1)}$ will be denoted by $p^{(i)}$. The overlap of $p^{(i)}$ and $p^{(i+1)}$ will be denoted by $q^{(i)}$.
\end{definition}

\begin{lemma}
\label{Opiwn-1}
The factor $p^{(i)}$ is an image of $w_{n-1}$ for $i \in \{0,1,2\}$.
\end{lemma}
\begin{proof}
Since $w_n^{(0)}$ overlaps with $w_n^{(2)}$, the assumptions of Lemma \ref{L_1_p} are satisfied with $w_n := w_n^{(0)}$ and $w_{n+1} := w_n^{(0,2)}$.

We obtain $w_n^{(0,2)} = w_{n-1}p_1^{-1}\vartheta_1(w_{n-1})p_2^{-1}\vartheta_2\vartheta_1(w_{n-1})\vartheta_2(p_1)^{-1}\vartheta_2(w_{n-1}),$ where $|p_1| = |p_2|$. We have, $w_n^{(0)} = w_{n-1}p_1^{-1}\vartheta_1(w_{n-1})$ and $w_n^{(1)} = \vartheta_1(w_{n-1})p_2^{-1}\vartheta_2\vartheta_1(w_{n-1})$. Hence, the overlap of $w_n^{(0)}$ and $w_n^{(1)}$ is $p^{(0)} = \vartheta_1(w_{n-1})$. It is readily seen that $p^{(1)}$ and $p^{(2)}$ are images of $p^{(0)}$, i.e., they are images of $w_{n-1}$.

\end{proof}

Our further considerations are divided into two cases. Either $p^{(0)}$ overlaps with $p^{(2)}$, which is equivalent to say that $w_n^{(0)}$ and $w_n^{(3)}$ overlap, or it does not. The first case is treated in Lemma \ref{Lqi}, the second one in Proposition \ref{L_2_spec} (it provides us with four new prefix substitution rules).

\begin{lemma}
\label{Lqi}
Suppose that exactly two pseudopalindromic prefixes were missed between $w_n$ and $w_{n+1}$ such that $w_n^{(0)}$ and $w_n^{(2)}$ overlap. Assume that the prefix of length $n$ of the directive bi-sequence $\DT$ is normalized. Furthermore, suppose that $p^{(0)}$ overlaps with $p^{(2)}$. Then $q^{(0)}$ and $q^{(1)}$ are both images of $w_{n-2}$.
\end{lemma}
\begin{proof}
It is easily seen that the word $p^{(0,2)}$ is a prefix of some generalized pseudostandard word. Furthermore, if we make the pseudopalindromic closure $(p^{(0)}\delta)^{\vartheta}$ with $\delta$ and $\vartheta$ satisfying $w_{n+1}=(w_n\delta)^{\vartheta}$, we obtain the word $p^{(0,2)}$, and the word $p^{(0,1)}$ was missed. Thus, $p^{(0)}$ and $p^{(0,2)}$ satisfy the assumptions of Lemma \ref{L_1_p}.

Now, $p^{(0,2)} = w_{n-1}' p_1^{-1}\vartheta_1'(w_{n-1}')p_2^{-1}\vartheta_2'\vartheta_1'(w_{n-1}')\vartheta_2'(p_1)^{-1}\vartheta_2'(w_{n-1}'),$ where $|p_1| = |p_2|$. Using the same arguments as in the proof of Lemma \ref{Opiwn-1}, the overlap of $p^{(0)}$ and $p^{(1)}$ is equal to $\vartheta_1'(w_{n-1}')$. Since $p^{(0)}$ is an image of $w_{n-1}$, then $\vartheta_1'(w_{n-1}')$ is an image of $w_{n-2}$. Thus $q^{(0)}$ is an image of $w_{n-2}$. Since $q^{(1)}$ is an image of $q^{(0)}$, it is an image of $w_{n-2}$, too.
\end{proof}

\begin{proposition}
\label{L_2_spec}
Suppose that exactly two pseudopalindromic prefixes were missed between $w_n$ and $w_{n+1}$ such that $w_n^{(0)}$ and $w_n^{(2)}$ overlap. Assume that the prefix of length $n$ of the directive bi-sequence $\DT$ is normalized. Furthermore, suppose that $p^{(0)}$ does not overlap with $p^{(2)}$. Then either $q^{(0)}$ and $q^{(1)}$ are both images of $w_{n-2}$, or $w_{n+1}$ is of one of the following forms:
\begin{itemize}
\item $(ij)^l i (ki)^l k (jk)^l j (ij)^l i (ki)^l k (jk)^l$

\item $i^lj^{l+1}k^{l+1}i^{l+1}j^{l+1}k^l$

\item $(ij)^l ik(jk)^l ji(ki)^l kj(ij)^l ik (jk)^l ji (ki)^l$

\item $ij(kjij)^{l-1} kjik(ijik)^{l-1} ijikjk(ikjk)^{l-1} ikji(jkji)^{l-1}\ldots$ \\
$\ldots jkjiki(jiki)^{l-1} ji kj(kikj)^{l-1}ki$
\end{itemize}
\end{proposition}
\begin{proof} If $p^{(0)}$ and $p^{(2)}$ do not overlap, then $w_{n-1}^{(0)}$ and $w_{n-1}^{(2)}$ do no overlap, neither. Therefore, if we put $w_n:=w_{n-1}$ and $w_{n+1}:=w_n^{(0,1)}$, then the assumptions of Proposition \ref{L_1_1} or Proposition \ref{L_1_2} are satisfied. (Notice that $w_{n-1}^{(0,2)}=w_n^{(0,1)}$.)

If $w_n^{(0,1)} = w_{n-1} i \T_2 (w_{n-1})$ and $w_{n-1} = w_{n-2} j \T_1 (w_{n-2})$, then
$$w_n^{(0,1)} = w_{n-2}j  \T_1 (w_{n-2})i \T_2 \T_1 (w_{n-2}) \T_2 (j) \T_2 (w_{n-2})$$
and $q^{(i)}$ is an image of $w_{n-2}$. Similarly in the case where $w_n^{(0,1)} = w_{n-1} ij \T_2 (w_{n-1})$ and $w_{n-1}=w_{n-2}kl\vartheta_1(w_{n-2})$.

Now suppose that $w_n^{(0,1)}$ has one of the rest of the forms of $w_{n+1}$ in Proposition \ref{L_1_1} or Proposition \ref{L_1_2}.
Since we consider the situation of two missed pseudopalindromic prefixes between $w_n$ and $w_{n+1}$, by Corollary \ref{CstairsType}, there exist $i$, $j$, $k$ pairwise different such that $w_n$ and $w_{n+1}$ are $E_i$-palindromes, $w_n^{(0,1)}=E_j(w_n^{(0,1)})$, and $w_n^{(0,2)}=E_k(w_n^{(0,2)})$. Therefore, we can exclude the cases where $w_n=R(w_n)$ or $w_n^{(0,1)}=R(w_n^{(0,1)})$. The remaining cases are:

\begin{itemize}

%
%
%
%
%
%
%
%
%
%
%
%

\item $w_{n-1} = (ij)^l i(ki)^l$:

$w_n^{(0,1)} = (ij)^l i(ki)^l k(jk)^l j(ij)^l$ is an $E_k$-palindrome.

$w_n = (ij)^l i(ki)^l k(jk)^l$ is an $E_j$-palindrome.

Now, the $E_j$-palindromic closure of $w_nj$ is $$(w_nj)^{E_j} = (ij)^l i (ki)^l k (jk)^l j (ij)^l i (ki)^l k (jk)^l = w_{n+1}$$ and we missed the $E_k$-palindrome $w_n^{(0,1)}$ and the $E_i$-palindrome \\$w_n^{(0,2)} =  (ij)^l i (ki)^l k (jk)^l j (ij)^l i (ki)^l$.

%
%

%
%

%
%
%

\item $w_{n-1} = i^lj^l$:

$w_n^{(0,1)} = i^lj^{l+1}k^{l+1}i^{l}$ is an $E_i$-palindrome.

$w_n = i^lj^{l+1}k^l$ is an $E_j$-palindrome.

Now, the $E_j$-palindromic closure of $w_nk$ is $$(w_nk)^{E_j} = i^lj^{l+1}k^{l+1}i^{l+1}j^{l+1}k^l = w_{n+1}$$ and we missed the $E_i$-palindrome $w_n^{(0,1)}$ and the $E_k$-palindrome \\$w_n^{(0,2)} =  i^lj^{l+1}k^{l+1}i^{l+1}j^l$.

\item $w_{n-1} = (ij)^l ik (jk)^l$:

$w_n^{(0,1)} = (ij)^l ik(jk)^l ji(ki)^l kj(ij)^l$ is an $E_k$-palindrome.

$w_n = (ij)^l ik(jk)^l ji(ki)^l$ is an $E_i$-palindrome.

Now, the $E_i$-palindromic closure of $w_nk$ is $$(w_nk)^{E_i} = (ij)^l ik(jk)^l ji(ki)^l kj(ij)^l ik (jk)^l ji (ki)^l = w_{n+1}$$ and we missed the $E_k$-palindrome $w_n^{(0,1)}$ and the $E_j$-palindrome \\$w_n^{(0,2)} =  (ij)^l ik(jk)^l ji(ki)^l kj(ij)^l ik (jk)^l$.

%
%

\item $w_{n-1} = ij(kjij)^{l-1} kjik(ijik)^{l-1} ij$:

$w_n^{(0,1)} = ij(kjij)^{l-1} kjik(ijik)^{l-1} ijikjk(ikjk)^{l-1} ikji(jkji)^{l-1} jk =E_j(w_n^{(0,1)})$.

$w_n = ij(kjij)^{l-1} kjik(ijik)^{l-1} ijikjk (ikjk)^{l-1} i$ is an $E_i$-palindrome.

Now, the $E_i$-palindromic closure of $w_nk$ is

\begin{align*}
(w_nk)^{E_i} &=& ij(kjij)^{l-1} kjik(ijik)^{l-1} ijikjk(ikjk)^{l-1} ikji(jkji)^{l-1} \ldots\\
   & &\ldots  jkjiki(jiki)^{l-1}jikj(kikj)^{l-1}ki
\end{align*}
and we missed the $E_j$-palindrome $w_n^{(0,1)}$ and the $E_k$-palindrome \\$w_n^{(0,2)} =  ij(kjij)^{l-1} kjik(ijik)^{l-1} ijikjk(ikjk)^{l-1} ikji(jkji)^{l-1} jkjiki(jiki)^{l-1} j$.

\end{itemize}


\end{proof}

\subsubsection{Normalization rules}
\noindent {\bf Prefix rules}

\noindent The following prefix substitution rules for the case of two missed pseudopalindromic prefixes between $w_n$ and $w_{n+1}$ are deduced from Propositions \ref{Ospecwijw}, \ref{Ospecwiw}, and \ref{L_2_spec} and its proof:

\begin{itemize}
\item $w_{n+1} = ijki$:
\begin{prefixrule}\
\label{p2}
 (ij, E_iE_i) \rightarrow (ijki, E_iE_kE_jE_i),
\end{prefixrule}

\item $w_{n+1} = ijjkkiijjk$:
\begin{prefixrule}\
\label{p5}
 (ijjk, E_iE_kE_jE_j) \rightarrow (ijjkij, E_iE_kE_jE_iE_kE_j).
\end{prefixrule}

\item $w_{n+1} = ijkij$:
\begin{prefixrule}\
\label{p14}
 (ijk,E_iE_kE_k) \rightarrow (ijkij, E_iE_kE_jE_iE_k).
\end{prefixrule}

\item $w_{n+1} = (ij)^l i (ki)^l k (jk)^l j (ij)^l i (ki)^l k (jk)^l$:
\begin{prefixrule}\
\label{p32}
 (i(ji)^{l}kkj, E_i(E_kR)^{l} E_iE_jE_j)\rightarrow (i(ji)^{l} kkjik, E_i(E_kR)^{l}E_iE_jE_kE_iE_j).
\end{prefixrule}

\item $w_{n+1} = i^lj^{l+1}k^{l+1}i^{l+1}j^{l+1}k^l$:
\begin{prefixrule}\
\label{p33}
 (i^{l} jjk, E_i^{l} E_kE_jE_j) \rightarrow (i^{l}jjkij, E_i^{l}E_kE_jE_iE_kE_j).
\end{prefixrule}

\item $w_{n+1} = (ij)^l ik(jk)^l ji(ki)^l kj(ij)^l ik (jk)^l ji (ki)^l$:
\begin{prefixrule}\
\label{p34}
 (i(ji)^{l} kjk, E_i(E_kR)^{l} E_jE_iE_i) \rightarrow (i(ji)^{l}kjkij, E_i(E_kR)^{l}E_jE_iE_kE_jE_i).
\end{prefixrule}

\item $w_{n+1} = ij(kjij)^{l-1} kjik(ijik)^{l-1} ijikjk(ikjk)^{l-1} ikji(jkji)^{l-1} \ldots$ \\
\hspace*{\fill}$\ldots jkjiki(jiki)^{l-1} jikj(kikj)^{l-1}ki$:\hspace{1em}

\begin{prefixrule}\
\label{p35}
\begin{split}
 (ijk(jj)^{l-1}jkik, E_iE_kE_j(RE_j)^{l-1}RE_kE_iE_i) \rightarrow \\(ijk(jj)^{l-1}jkikji, RE_iE_kE_j(RE_j)^{l-1} E_kE_iE_jE_kE_i).
\end{split}
\end{prefixrule}

\end{itemize}

\noindent {\bf Factor rules}

\noindent The next theorem concludes the section concerning two pseudopalindromic prefixes being missed between $w_n$ and $w_{n+1}$. The last factor substitution rule is obtained.
\begin{theorem}
\label{nonprefix2}
Let $(\Delta,\Theta) = (\delta_1\delta_2\ldots, \vartheta_1\vartheta_2\ldots)$ be a directive bi-sequence having a normalized prefix of length $n$. Moreover, let the prefix of $(\Delta,\Theta)$ of length $n+1$ be different from any prefix on the left side of the prefix rules from \eqref{p1} to \eqref{p35}. Then there are exactly two missed pseudopalindromic prefixes between $w_n$ and $w_{n+1}$ if, and only if, $(\delta_{n-2}\delta_{n-1}\delta_n\delta_{n+1}$, $\vartheta_{n-2}\vartheta_{n-1}\vartheta_n\vartheta_{n+1})$ is of the form $(ab_1b_2b_3, E_iE_jE_kE_k)$, where $E_i(b_1) = E_j(b_2) = E_k(b_3)$.
We obtain the last factor substitution rule:
\begin{itemize}
\item
$(ab_1b_2b_3, E_iE_jE_kE_k) \rightarrow (ab_1b_2b_3b_1b_2, E_iE_jE_kE_iE_jE_k)$, where $E_i(b_1) = E_j(b_2) = E_k(b_3)$.
\end{itemize}
\end{theorem}
\begin{proof}

$(\Rightarrow)$:
Two pseudopalindromic prefixes were missed between $w_n$ and $w_{n+1}$. Thus, by Corollary \ref{CstairsType}, $w_n = E_k(w_n)$, $w_n^{(0,1)} = E_i( w_n^{(0,1)})$, $w_n^{(0,2)} = E_j (w_n^{(0,2)})$, and $w_{n+1} = E_k(w_{n+1})$ for pairwise different $i,j,k$.

The word $p^{(0)}$ is a central factor of $w_n^{(0,1)}$, thus it is an $E_i$-palindrome, too, and, moreover, $p^{(0)}= E_k(w_{n-1})$ by Lemma \ref{Opiwn-1} and its proof. Consequently, by Observation \ref{o_natureofpalonpal}, $w_{n-1} = E_j(w_{n-1})$. Analogously, we deduce that $w_{n-2} = E_i(w_{n-2})$.

	\begin{figure}[ht!]
		\begin{center}
			\includegraphics[scale=1.1]{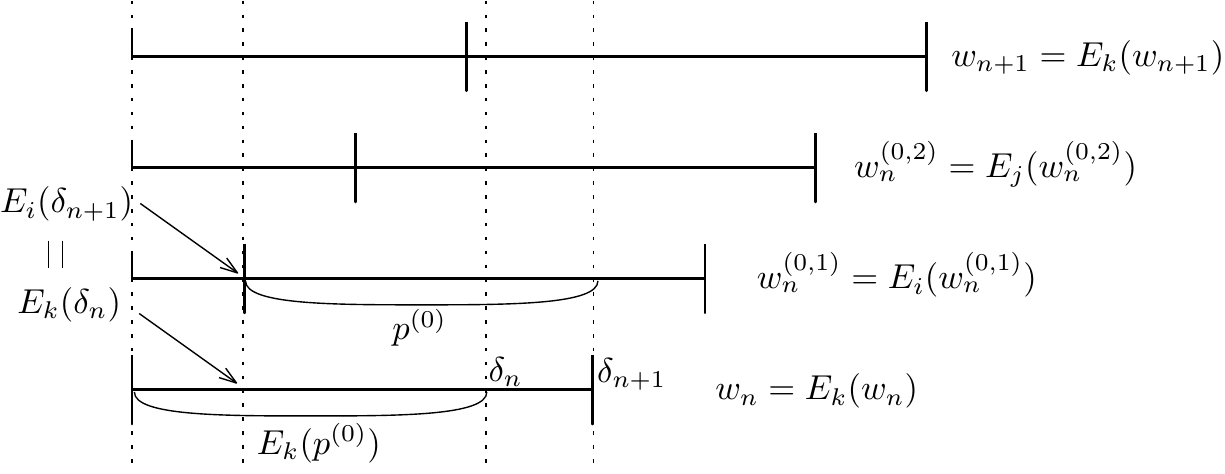}
		\end{center}
		\caption{Illustration of the relation between $\delta_n$ and $\delta_{n+1}$.}
		\label{OT21}
	\end{figure}

We will now find the relations of $\delta_{n-1}$, $\delta_n$, and $\delta_{n+1}$. For a better understanding see Figure \ref{OT21}. The prefix $w_n$ (and also $p^{(0)}$) is followed by the letter $\delta_{n+1}$. Since $p^{(0)}$ is a central factor of the $E_i$-palindrome $w_n^{(0,1)}$, $p^{(0)}$ is preceded by the letter $E_i(\delta_{n+1})$. In addition, the prefix $w_{n-1} = E_k(p^{(0)})$ of $w_n$ is followed by the letter $\delta_n$. Since $w_n$ is an $E_k$-palindrome, $p^{(0)}$ is also preceded by the letter $E_k(\delta_n)$. We obtain the equality $E_k(\delta_n) = E_i(\delta_{n+1})$.

Using the suffix $q^{(0)}$ of the word $w_n$ and the $E_j$-palindrome $w_n^{(0,2)}$, the equality $E_k(\delta_{n-1}) = E_j(\delta_{n+1})$ can be deduced analogously. Overall, we obtain:
$$E_i(\delta_{n-1}) = E_iE_kE_j(\delta_{n+1}) = E_k(\delta_{n+1}) = E_kE_iE_k(\delta_n) = E_j(\delta_n).$$

$(\Leftarrow)$:
Knowing the form of $(\delta_{n-2}\delta_{n-1}\delta_n\delta_{n+1}, \vartheta_{n-2}\vartheta_{n-1}\vartheta_n\vartheta_{n+1})$, we can easily deduce that $E_j(w_{n-2})$ is the longest $E_k$-palindromic suffix used when constructing $w_n$. Thus
\begin{equation} \label{eq:wn}
w_n = w_{n-1}(E_j(w_{n-2}))^{-1}E_k(w_{n-1}).
\end{equation}

	\begin{figure}[ht!]
		\begin{center}
			\includegraphics[scale=1.1]{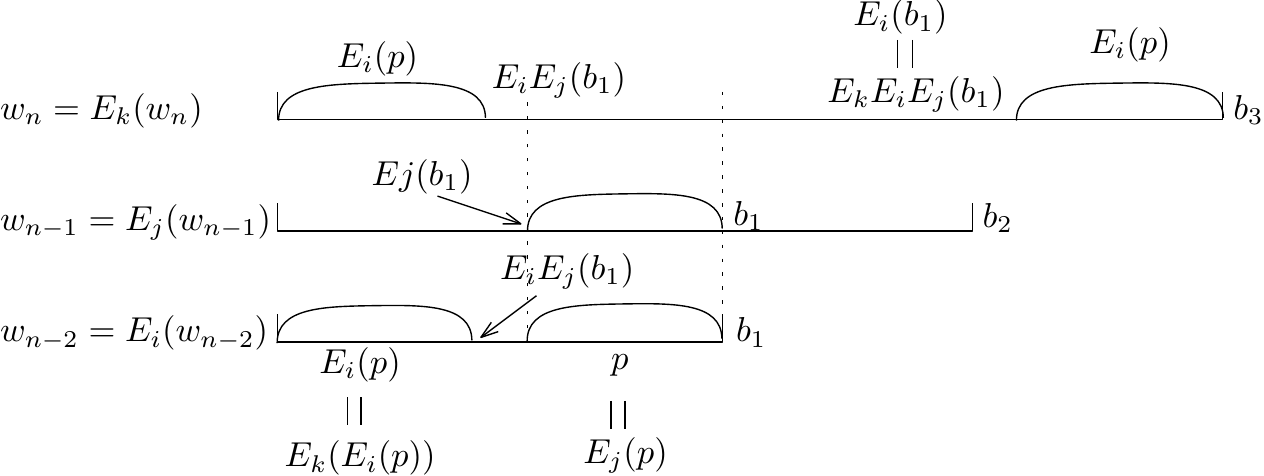}
		\end{center}
		\caption{Finding the longest $E_k$-palindromic suffix of $w_nb_3$.}
		\label{O1T22}
	\end{figure}
	
Now, let us look for the longest $E_k$-palindromic suffix of $w_nb_3$, see Figure \ref{O1T22}. When constructing $w_{n-1}$, we looked for the longest suitable $E_j$-palindromic suffix of $w_{n-2}$ -- let $p$ denote this suffix. Hence, $w_{n-1} = w_{n-2}p^{-1}E_j(w_{n-2})$, from which we have
$$p = (w_{n-2}^{-1}w_{n-1}E_j(w_{n-2})^{-1})^{-1}.$$
Since $w_{n-2}$ is an $E_i$-palindrome, $E_i(p)$ is an $E_k$-palindromic prefix of $w_{n-2}$. Since $p$ is followed by $b_1$, it is preceded by $E_j(b_1)$, consequently, $E_i(p)E_iE_j(b_1)$ is a prefix of $w_n$.

Applying the antimorphism $E_i$ on the previous equation, we obtain:
\begin{equation} \label{eq:sufix}
E_i(p) = ((E_k(w_{n-2}))^{-1}E_i(w_{n-1})(w_{n-2})^{-1})^{-1}.
\end{equation}
Since $w_n$ is an $E_k$-palindrome and $E_i(p)$ is an $E_k$-palindromic prefix of $w_n$, then $E_i(p)$ is also an $E_k$-palindromic suffix of $w_n$. Moreover, the suffix $E_i(p)$ of $w_n$ is preceded by the letter $E_kE_iE_j(b_1) = E_i(b_1)$ and followed by the letter $b_3$. Thus, $E_i(p)$ is the longest suitable $E_k$-palindromic suffix of $w_n$, where we used the fact that the prefix of length $n$ of the directive bi-sequence is normalized, and the assumption that $E_i(b_1) = E_k(b_3)$.
Combining the equations~\eqref{eq:wn} and \eqref{eq:sufix}, we obtain:
\begin{align} \label{eq:wn1}
\begin{split}
w_{n+1}  = \;\, & w_{n-1}(E_j(w_{n-2}))^{-1}E_k(w_{n-1})(E_k(w_{n-2}))^{-1}E_i(w_{n-1})(w_{n-2})^{-1} \\
 & w_{n-1}(E_j(w_{n-2}))^{-1}E_k(w_{n-1}).
\end{split}
\end{align}
The two missed pseudopalindromic prefixes can be easily found in \eqref{eq:wn1}:
\begin{align}
w_n^{(0,1)}=&\;\, w_{n-1}(E_j(w_{n-2}))^{-1}E_k(w_{n-1})(E_k(w_{n-2}))^{-1}E_i(w_{n-1})=E_i(w_n^{(0,1)}), \nonumber \\
w_n^{(0,2)}=&\;\, w_{n-1}(E_j(w_{n-2}))^{-1}E_k(w_{n-1})(E_k(w_{n-2}))^{-1}E_i(w_{n-1}) \nonumber \\
& \;\,(w_{n-2})^{-1}w_{n-1}=E_j(w_n^{(0,2)}). \nonumber
\end{align}
Finally, we obtain the factor substitution rule knowing the equalities for $\delta_{n-1}$, $\delta_n$, and $\delta_{n+1}$, and the missed pseudopalindromic prefixes:
\begin{itemize}
\item
$(ab_1b_2b_3, E_iE_jE_kE_k) \rightarrow (ab_1b_2b_3b_1b_2, E_iE_jE_kE_iE_jE_k)$, where $E_i(b_1) = E_j(b_2) = E_k(b_3)$.
\end{itemize}
\end{proof}

\subsection{Final algorithm}\label{FinalAlgorithm}

Before presenting the algorithm, we will compile the final list of prefix substitution rules. The rules \eqref{p6}, \eqref{p7}, and \eqref{p5} were removed because they were special cases of the rules \eqref{p25}, \eqref{p23}, and \eqref{p33}, respectively. The rule \eqref{p10} was merged with the rule \eqref{p16}, and the rule \eqref{p11} was merged with the rule \eqref{p13}. Moreover, for a better readability, the index $l$ was changed to $n+1$. Thus, $n$ can take any non-negative integer value. The condition in the rule \eqref{p3} was removed by incrementing the index by one.

\begin{definition} \label{def:rules}
\textit{A normalization prefix rule} is one of the following set of prefix substitution rules:
\begin{enumerate}
\item
$(i^{n+1}, E_i^nE_k) \rightarrow (i^{n+1}j, E_i^{n+1}E_k)$,

\item
$(i(ji)^{n+1}j, E_i(E_kR)^{n+1}E_i) \rightarrow (i(ji)^{n+1}jk, E_i(E_kR)^{n+1}E_kE_i)$,
\item
$((ij)^{n+1}i, E_iE_k(RE_k)^nE_j) \rightarrow ((ij)^{n+1}ik, E_iE_k(RE_k)^nRE_j)$,
\item
$(ijk(jj)^nj, E_iE_kE_j(RE_j)^nE_k) \rightarrow (ijk(jj)^njk, E_iE_kE_j(RE_j)^nRE_k)$,
\item
$(ijkj(jj)^nj, E_iE_kE_jR(E_jR)^nE_i) \rightarrow (ijkj(jj)^nji, E_iE_kE_jR(E_jR)^nE_jE_i)$,
\item
$(ij(ij)^ni, E_iE_k(RE_k)^nE_i) \rightarrow (ij(ij)^nik, E_iE_k(RE_k)^nRE_i)$,
\item
$(i(ji)^nj,E_i(E_kR)^nE_j) \rightarrow (i(ji)^njk, E_i(E_kR)^nE_kE_j)$,
\item
$((ijk)^nijk, (E_iE_kE_j)^nE_iE_kR) \rightarrow ((ijk)^{n+1}j, (E_iE_kE_j)^{n+1}R)$,
\item
$((ijk)^nij, (E_iE_kE_j)^nE_iR) \rightarrow ((ijk)^niji, (E_iE_kE_j)^nE_iE_kR)$,
\item
$((ijk)^{n+1}i, (E_iE_kE_j)^{n+1}R) \rightarrow ((ijk)^{n+1}ik, (E_iE_kE_j)^{n+1}E_iR)$,
\item
$((ijk)^{n+1}jk, (E_iE_kE_j)^{n+1}RR) \rightarrow ((ijk)^{n+1}jkk, (E_iE_kE_j)^{n+1}RE_kR)$,
\item
$((ijk)^{n+1}ikk, (E_iE_kE_j)^{n+1}E_iRR) \rightarrow ((ijk)^{n+1}ikki, (E_iE_kE_j)^{n+1}E_iRE_jR)$,
\item
$((ijk)^nijik, (E_iE_kE_j)^nE_iE_kRR) \rightarrow ((ijk)^nijikj, (E_iE_kE_j)^nE_iE_kRE_iR)$,
\item
$(i(ji)^njkj, E_i(E_kR)^nE_kE_jE_j) \rightarrow (i(ji)^njkjj, E_i(E_kR)^nE_kE_jRE_j)$,
\item
$(ij(ij)^nikk, E_iE_k(RE_k)^nRE_iE_k) \rightarrow (ij(ij)^nikkj, E_iE_k(RE_k)^nRE_iE_jE_k)$,
\item
$(i^{n+1}jj, E_i^{n+1}E_kE_k) \rightarrow (i^{n+1}jji, E_i^{n+1}E_kRE_k)$,
\item
$(ij(ij)^{n+1}kk, E_iE_k(RE_k)^{n+1}E_iE_i) \rightarrow (ij(ij)^{n+1}kkj, E_iE_k(RE_k)^{n+1}E_iRE_i)$,
\item
$(i^{n+1}jj, E_i^{n+1}E_kE_i) \rightarrow (i^{n+1}jjk, E_i^{n+1}E_kE_jE_i)$,
\item
$(ij(ij)^nikj, E_iE_k(RE_k)^nRE_jE_k) \rightarrow (ij(ij)^nikjk, E_iE_k(RE_k)^nRE_jE_iE_k)$,
\item
$(ijk(jj)^nik, E_iE_kE_j(RE_j)^nE_iE_i) \rightarrow (ijk(jj)^nikj, E_iE_kE_j(RE_j)^nE_iRE_i)$,
\item
$(ijk(jj)^njki, E_iE_kE_j(RE_j)^nRE_kE_j) \rightarrow \\ (ijk(jj)^njkik, E_iE_kE_j(RE_j)^nRE_kE_iE_j)$,
\item
$(i^{n+1}ji, E_i^{n+1} E_kE_k) \rightarrow (i^{n+1}jij, E_i^{n+1} E_kRE_k)$,
\item
$(i^{n+1} ijj, E_i^{n+1} E_iE_jE_j) \rightarrow (i^{n+1} ijjj, E_i^{n+1} E_iE_jRE_j)$,
\item
$(i^{n+1}ijk, E_i^{n+1}E_iE_iE_i) \rightarrow (i^{n+1}ijkj, E_i^{n+1}E_iE_iRE_i)$,

\item
$(ij, E_iE_i) \rightarrow (ijki, E_iE_kE_jE_i)$,
\item
$(ijk,E_iE_kE_k) \rightarrow (ijkij, E_iE_kE_jE_iE_k)$,
\item
$(i(ji)^{n+1}kkj, E_i(E_kR)^{n+1} E_iE_jE_j)\rightarrow \\ (i(ji)^{n+1} kkjik, E_i(E_kR)^{n+1}1E_iE_jE_kE_iE_j))$,
\item
$(i^{n+1} jjk, E_i^{n+1} E_kE_jE_j) \rightarrow (i^{n+1}jjkik, E_i^{n+1}E_kE_jE_iE_kE_j))$,
\item
$(i(ji)^{n+1} kjk, E_i(E_kR)^{n+1} E_jE_iE_i) \rightarrow (i(ji)^{n+1}kjkij, E_i(E_kR)^{n+1}E_jE_iE_kE_jE_i)$,
\item
$(ijk(jj)^{n}jkik, E_iE_kE_j(RE_j)^nRE_kE_iE_i) \rightarrow \\ (ijk(jj)^{n}jkikji, E_iE_kE_j(RE_j)^n E_kE_iE_jE_kE_i)$.

\end{enumerate}
\end{definition}
Furthermore, the factor substitution rules from Theorem \ref{nonprefix1} and \ref{nonprefix2} will be reminded:
\begin{definition}
\label{factorrules}
\textit{A normalization factor rule} is one of the following set of factor substitution rules:
\begin{enumerate}
\item
$(ab_1b_2, RE_iE_i) \rightarrow (ab_1b_2b_1, RE_iRE_i),  \text{ where } b_1 = E_i(b_2)$,
\item
$(ab_1b_2, E_iRR) \rightarrow (ab_1b_2b_1, E_iRE_iR), \text{ where } b_1 = E_i(b_2)$,
\item
$(ab_1b_2, E_iE_jE_i) \rightarrow (ab_1b_2E_iE_j(b_1), E_iE_jE_kE_i), \text{ where } E_i(b_1)= E_j(b_2)$,
\item
$(ab_1b_2b_3, E_iE_jE_kE_k) \rightarrow (ab_1b_2b_3b_1b_2, E_iE_jE_kE_iE_jE_k)$, where $E_i(b_1) = E_j(b_2) \\ = E_k(b_3)$.
\end{enumerate}
\end{definition}

Let $\DT$ be any ternary directive bi-sequence. The normalization algorithm of $\DT$ will be described in the sequel:

\begin{enumerate}
\item Find the length $l$ of the longest prefix of $\DT$ such that $\Delta$ contains only the letter $i$ and $\Theta$ contains only the antimorphisms $R$ and $E_i$. Modify the prefix of $\Theta$ to $E_i^l$.

\item Check whether some normalization prefix rules of Definition \ref{def:rules} or some normalization factor rules of Definition \ref{factorrules} are applicable. If there are none, $\DT$ is normalized. If there are any, apply the rule that can be used on the shortest prefix of $\DT$. Repeat step 2 until $\DT$ is normalized.
\end{enumerate}

The second step of the algorithm does not necessarily end after a final number of steps, but with every step, a strictly longer normalized prefix of $\DT$ is obtained.

Finally, an example illustrates the algorithm.
\begin{example}
Let $\DT$ be $(010221011^{\omega}, RRE_0E_2E_1E_2E_1E_0E_2^{\omega})$. The normalization algorithm proceeds in the following steps:
\begin{itemize}
	\item First, changing the prefix of $\Theta$: $(010221011^{\omega}, E_0RE_0E_2E_1E_2E_1E_0E_2^{\omega})$.
	\item Applying the normalization prefix rule 9:\\ $((012)^0 01, (E_0E_2E_1)^0 E_0 R) \to$ $((012)^0 010, (E_0E_2E_1)^0 E_0 E_2R)$:
	$$ (0100221011^{\omega}, E_0E_2RE_0E_2E_1E_2E_1E_0E_2^{\omega}).$$
	\item Applying the normalization factor rule 3: $(210, E_1 E_2 E_1) \rightarrow (2102, E_1E_2E_0E_1)$:
	$$ (01002210211^{\omega}, E_0E_2RE_0E_2E_1E_2E_0E_1E_0E_2^{\omega}). $$
	
	\item Applying the normalization factor rule 3: $(021, E_0 E_1 E_0) \rightarrow (0210, E_0E_1E_2E_0)$:
	$$ (0100221021011^{\omega}, E_0E_2RE_0E_2E_1E_2E_0E_1E_2E_0E_2^{\omega}). $$
\end{itemize}

None of the rules can be applied further on, therefore \\ $(0100221021011^{\omega}$,
$E_0E_2RE_0E_2E_1E_2E_0E_1E_2E_0E_2^{\omega})$ is the normalized bi-sequence of $\mbu \DT$.
\end{example}

\section{Implementation} \label{Implementation}
Alongside our theoretical work, we implemented and tested the new normalization algorithm presented in Section~\ref{FinalAlgorithm}. The documented code and examples are publicly available at

\begin{center}
https://github.com/velkater/tgpc
\end{center}

Comparing the new normalization algorithm to a naive normalization algorithm helped to obtain the final set of normalization rules.

\subsection{Implementation of the normalization algorithm}
In this section, the key aspects of our implementation are presented. We implemented the normalization algorithm as a Python 3 module called \texttt{tgpc}, standing for ``ternary generalized pseudopalindromic closures'', that can be found on the provided link.

The new normalization algorithm of a ternary directive bi-sequence $\DT$ is implemented in the method \texttt{normalize} of the object \texttt{Normalizer012}. The input is a string representing $\Delta$ composed of the letters $0$, $1$, $2$, and a string representing $\Theta$ composed of the letters $R$, $0$, $1$, $2$, standing for the involutory antimorphisms $R$, $E_0$, $E_1$, and $E_2$.

\subsection{Preprocessing of the directive bi-sequence}
In order to make the algorithm easier to read and write, we decided to work only with generalized pseudostandard words that have $0$ as the first occurring letter, $1$ as the second one, and $2$ as the third one. Naturally, we want our algorithm to work correctly for all directive bi-sequences. That is why, at the beginning of the function \texttt{normalize}, the function \texttt{\_change\_letters\_order} processes the given directive bi-sequence. The resulting bi-sequence $\Delta'$ and $\Theta'$ generates the same generalized pseudostandard word, except that the letters $0$, $1$, and $2$ appear first in this order.

It is easily seen that the processing of the bi-sequence described above can be done without having to compute the generated generalized pseudostandard word. First, we want to change the first letter to $0$: if the first letter appearing in $\Delta$ is not $0$ but $a \in \{1,2\}$, then we have to substitute $0 \rightarrow a$ in both $\Delta$ and $\Theta$. Now, while the prefix of $\DT$ is $(0^l, E_0^l)$, the order of letters cannot be decided. Let $\delta$ and $\vartheta$ be the first letters following the longest prefix of the form $(0^l, E_0^l)$. If $\delta$ is $0$ and $\vartheta$ is $E_2$, the resulting word has the desired letter order. If $\delta$ is $0$ and $\vartheta$ is $E_1$, then we have to apply the substitution $\{1 \rightarrow 2, 2 \rightarrow 1\}$ to both $\Delta$ and $\Theta$.  If $\delta$ is $1$, then the resulting word has also the desired letter order. And, finally, if $\delta$ is $2$, the substitution $\{1 \rightarrow 2, 2 \rightarrow 1\}$ has to be applied.

At the end of the algorithm, a reverse substitution is applied to the new normalized directive bi-sequence to obtain the original order of letters.

\subsection{Normalization algorithm}
The preprocessed directive bi-sequence $(\Delta', \Theta') = (\delta_1 \delta_2 \ldots, \vartheta_1 \vartheta_2 \ldots)$ is represented as the string $\delta_1 \vartheta_1 \delta_2 \vartheta_2 \ldots$. The normalization algorithm from Section~\ref{FinalAlgorithm} can be now applied to $(\Delta', \Theta')$:

\begin{enumerate}
\item First, the private function \texttt{\_initial\_normalization(biseq)} finds the longest prefix of $(\Delta', \Theta')$ such that $\Delta'$ contains only the letter $0$ and $\Theta'$ contains only the antimorphisms $R$ and $E_0$ using a regular expression, and replaces all occurrences of $R$ by $E_0$ inside $\Theta'$.

\item The private \texttt{\_Normalization012\_rules\_checker} object is used to check if some normalization rule is applicable. If it is, it returns the next normalization rule to apply. The rule is applied and the newly corrected directive bi-sequence is presented again to the \texttt{\_Normalization012\_rules\_checker}. This process continues until no normalization rule is applicable.

\end{enumerate}
Let $(\widetilde{\Delta}, \widetilde{\Theta})$ be the normalized directive bi-sequence of $\DT$.
The method \texttt{normalize} returns the string representing $\widetilde{\Delta}$, the string representing $\widetilde{\Theta}$, and a boolean \texttt{notchanged}, which is equal to \texttt{true} if the sequence $\DT$ was already normalized, and false otherwise.

We will now briefly describe the \texttt{\_Normalization012\_rules\_checker} object. Its role is to check if a normalization rule can be applied on a given directive bi-sequence $\delta_1 \vartheta_1 \delta_2 \vartheta_2 \ldots$, to decide what is the next rule to apply, and to return the corresponding correction and the position where to apply it. This work is done by its public method \texttt{find\_applicable\_rule}.

The next simple observation explains the logic of this function:
\begin{observation}
Only one normalization prefix rule can be applied on a directive bi-sequence $\DT$. Moreover, if a normalization prefix rule can be applied on $\DT$, then no normalization factor rule can be applied on $\DT$.
\end{observation}
\begin{proof}
The statement is a direct corollary of the fact that the left sides of the normalization prefix rules are not normalized, but their prefixes without the last letter in each directive sequence are normalized.
\end{proof}
Note that the observation does not say that if we apply a normalization prefix rule, then no other normalization rule can be applied to the modified directive bi-sequence. This is not true in general.

First, the function \texttt{find\_applicable\_rule} checks if a normalization prefix rule is applicable. If it is, it returns the correction and the position to apply it. If not, it looks through normalization factor rules and finds the next factor normalization rule to be applied. It also computes the correction of the factor rule. Finally, it returns the correction and the position to be corrected. If no normalization rule is applicable, it returns \texttt{None}.

The normalization prefix rules and the normalization factor rules are represented by regular expressions to be matched on the directive bi-sequence $\delta_1 \vartheta_1 \delta_2 \vartheta_2 \ldots$. The normalization prefix rules given in Definition \ref{def:rules} are written so that the order of the letters in the resulting word is always $i$, $j$, and $k$. Since we have a fixed letter order, the prefix rules can be obtained by taking the $30$ prefix rules, replacing $i$ by $0$, $j$ by $1$, and $k$ by $2$ inside them, and finding their corresponding regular expressions. Here are the first regular expressions for the normalization prefix rules as an example:

\definecolor{keywords}{RGB}{0,46,184}
\definecolor{comments}{RGB}{0,0,113}
\definecolor{red}{RGB}{160,0,0}
\definecolor{green}{RGB}{0,100,0}
\definecolor{black}{RGB}{0,0,0}
\definecolor{navy}{RGB}{0,0,128}

\lstset{language=Python,
	basicstyle=\ttfamily\small,
	keywordstyle=\color{green},
	commentstyle=\color{red},
	stringstyle=\color{Violet},
	showstringspaces=false,
	identifierstyle=\color{navy},
	procnamekeys={def,class},
	procnamestyle=\bfseries\color{RoyalBlue}}
	
\begin{lstlisting}
 _bad_prefixes_and_correction = (
            ("(00)*02", "0012", 1),
            ("0010", "122100", 2),
            ("00(120R)+10", "1220", 3),
            ("0012(0R12)*01", "0R21", 4),
            ("001221(1R11)*12", "1R22", 5),
            ("0012211R(111R)*10", "1100", 6),
            ...
\end{lstlisting}

For example, the first normalization rule is represented as \texttt{"(00)*02"} corresponding to the first prefix rule $(0^{n+1}, E_0^nE_2) \rightarrow (0^{n+1}1, E_0^{n+1}E_2)$. The substring \texttt{(00)*} means that the factor $(0, E_0)$ can occur $0$, $1$ or more times and then it has to be followed by $(0, E_2)$. Each normalization prefix rule is also followed by the correction to apply on the last two letters of the matched string. Here, \texttt{02} is replaced by \texttt{0012}, which produces exactly  $(0^{n+1}1, E_0^{n+1}E_2)$.

The left sides of the normalization factor rules are generated inside the private function \texttt{\_generate\_factor\_rules} that finds all possibilities for each of the four rules. For example, the first possible forms of the left side of the first factor rule $(ab_1b_2, RE_iE_i) \rightarrow (ab_1b_2b_1, RE_iRE_i),  b_1 = E_i(b_2)$, are as follows:	
\begin{lstlisting}
['0R0000', '0R2101', '0R1202', '0R2010', '0R1111', ...
\end{lstlisting}
Or, in a more readable way:
\begin{lstlisting}
[['000', 'R00'], ['020', 'R11'], ['010', 'R22'], ['021', 'R00'],
['011', 'R11'], ['001', 'R22'], ...
\end{lstlisting}

The correction of the normalization factor rules is computed during the normalization process based on the right sides of the normalization factor rules given in Definition \ref{factorrules}.

\subsection{Naive normalization algorithm}

Besides implementing the new algorithm, we also implemented a naive normalization algorithm in order to test and compare their results.

The naive normalization process is implemented in the public method \texttt{normalize} of the \texttt{NaiveNormalizer012} object. The naive algorithm normalizes the directive bi-sequence as anybody would:

First, it finds all prefixes $w_n$ obtained by a (finite) directive bi-sequence $\DT$. Then it takes the generalized pseudostandard word generated by $\DT$ and looks for pseudopalindromes among its prefixes. Then it checks whether the prefixes $w_n$ are all the pseudopalindromic prefixes or not.

While implementing the naive algorithm, several necessary functions were implemented. They can be also easily used independently, their names are self-explanatory:
\texttt{ is\_pal(seq)},
\texttt{is\_eipal(seq, i)},
\texttt{make\_pal\_closure(seq)},
\texttt{make\_eipal\_closure (seq, i)},
\texttt{make\_word012(delta, theta)}. We used those functions to find or test some of our theoretical results, especially those ones concerning the normalization process.

\section{Conclusion}
In this paper, we presented several new results. Let us summarize them and mention some problems that remain open.
\begin{enumerate}
\item We described how to recognize whether a directive bi-sequence of a ternary generalized pseudostandard word is normalized and we provided an algorithm for normalization.
\item An important part of this work consisted in implementation of the new normalization algorithm. Our implementation is available in a Python module with several other functions permitting to work with ternary generalized pseudostandard words.
\item The authors of \cite{pepa2} found a necessary and sufficient condition for the periodicity of ternary generalized pseudostandard words. Using the normalization algorithm, we plan to improve the result showing that knowledge of the normalized directive bi-sequence is not necessary to decide whether the generalized pseudostandard word is periodic or not.
\item Knowledge of the normalized form of any directive bi-sequence and thus of all pseudopalindromic prefixes of the corresponding ternary generalized pseudostandard word can be used, in the future, to derive more combinatorial properties of ternary generalized pseudostandard words, for instance to obtain some results on their factor complexity.
\end{enumerate}

\section*{Acknowledgements}
We would like to thank Štěpán Starosta for his useful comments and advice concerning in particular implementation of the normalization algorithm.

\noindent Funding: This work was supported by the grant $\text{CZ}.02.1.01/0.0/0.0/16\_019/0000778$.




\bibliographystyle{elsarticle-num}
%

\end{document}